\documentclass[a4paper,11pt]{amsart}

\usepackage{a4wide}

\usepackage{amssymb}
\usepackage{mathrsfs}

\renewcommand{\Re}{\operatorname{Re}}

\newcommand{\eq}{:=}

\newcommand{\grad}{\boldsymbol \nabla}
\renewcommand{\div}{\grad \cdot}
\newcommand{\curl}{\grad \times}

\newcommand{\ccurl}{\boldsymbol{\operatorname{curl}}}
\newcommand{\ddiv}{\operatorname{div}}
\newcommand{\ggrad}{\boldsymbol{\operatorname{grad}}}

\newcommand{\jmp}[1]{\,[\![#1]\!]}

\newcommand{\BE}{\boldsymbol E}

\newcommand{\BH}{\boldsymbol H}
\newcommand{\BI}{\boldsymbol I}
\newcommand{\BJ}{\boldsymbol J}
\newcommand{\BK}{\boldsymbol K}
\newcommand{\BL}{\boldsymbol L}

\newcommand{\BW}{\boldsymbol W}
\newcommand{\BX}{\boldsymbol X}

\newcommand{\ba}{\boldsymbol a}

\newcommand{\bd}{\boldsymbol d}
\newcommand{\be}{\boldsymbol e}

\newcommand{\bh}{\boldsymbol h}

\newcommand{\bj}{\boldsymbol j}

\newcommand{\bn}{\boldsymbol n}
\newcommand{\bo}{\boldsymbol o}
\newcommand{\bp}{\boldsymbol p}

\newcommand{\br}{\boldsymbol r}

\newcommand{\bu}{\boldsymbol u}
\newcommand{\bv}{\boldsymbol v}
\newcommand{\bw}{\boldsymbol w}
\newcommand{\bx}{\boldsymbol x}
\newcommand{\by}{\boldsymbol y}

\newcommand{\CF}{\mathcal F}

\newcommand{\CN}{\mathcal N}

\newcommand{\CP}{\mathcal P}
\newcommand{\CQ}{\mathcal Q}
\newcommand{\CR}{\mathcal R}

\newcommand{\CT}{\mathcal T}

\newcommand{\CV}{\mathcal V}

\newcommand{\LC}{\mathscr C}

\newcommand{\LL}{\mathscr L}

\newcommand{\LP}{\mathscr P}

\newcommand{\TI}{\textup I}

\newcommand{\BCN}{\boldsymbol{\CN}}

\newcommand{\BCP}{\boldsymbol{\CP}}

\newcommand{\BCR}{\boldsymbol{\CR}}

\newcommand{\BCT}{\boldsymbol{\CT}}

\newcommand{\BTI}{\pmb{\TI}}

\newcommand{\SlopeTriangle}[6]
{

    \pgfplotsextra
    {
        \pgfkeysgetvalue{/pgfplots/xmin}{\xmin}
        \pgfkeysgetvalue{/pgfplots/xmax}{\xmax}
        \pgfkeysgetvalue{/pgfplots/ymin}{\ymin}
        \pgfkeysgetvalue{/pgfplots/ymax}{\ymax}

        \pgfmathsetmacro{\xArel}{#1}
        \pgfmathsetmacro{\yArel}{#3}
        \pgfmathsetmacro{\xBrel}{#1-#2}
        \pgfmathsetmacro{\yBrel}{\yArel}
        \pgfmathsetmacro{\xCrel}{\xArel}

        \pgfmathsetmacro{\lnxB}{\xmin*(1-(#1-#2))+\xmax*(#1-#2)} 
        \pgfmathsetmacro{\lnxA}{\xmin*(1-#1)+\xmax*#1} 
        \pgfmathsetmacro{\lnyA}{\ymin*(1-#3)+\ymax*#3} 
        \pgfmathsetmacro{\lnyC}{\lnyA+#4*(\lnxA-\lnxB)}
        \pgfmathsetmacro{\yCrel}{\lnyC-\ymin)/(\ymax-\ymin)} 

        \coordinate (A) at (rel axis cs:\xArel,\yArel);
        \coordinate (B) at (rel axis cs:\xBrel,\yBrel);
        \coordinate (C) at (rel axis cs:\xCrel,\yCrel);

        \draw[#6]   (A)-- node[anchor=north] {#5}
                    (B)--
                    (C)--
                    cycle;
    }
}

\usepackage[foot]{amsaddr}

\usepackage{xcolor}

\usepackage{cancel}

\usepackage{todonotes}

\usepackage{hyperref}

\usepackage{subcaption}

\usepackage{tikz}
\usetikzlibrary{patterns}
\usepackage{pgfplots}

\usepackage{algorithm}
\usepackage[noend]{algpseudocode}

\algnewcommand\algorithmicforeach{\textbf{for each}}
\algdef{S}[FOR]{ForEach}[1]{\algorithmicforeach\ #1\ \algorithmicdo}

\newcommand{\bet}{\boldsymbol \eta}

\newcommand{\eps}{\varepsilon}

\newcommand{\ee}{\boldsymbol \eps}

\newcommand{\cc}{\boldsymbol \chi}

\newcommand{\al}{\boldsymbol \alpha}

\newcommand{\divh}{\grad_h \cdot}
\newcommand{\curlh}{\grad_h \times}

\newcommand{\BJe}{\BJ_{\rm e}}
\newcommand{\BKe}{\BK_{\rm e}}

\newcommand{\enorm}[1]{\left |\!\left |\!\left | #1 \right |\!\right |\!\right |}

\newcommand{\OO}{\Omega}
\newcommand{\OOm}{{\OO_{\rm m}}}
\newcommand{\OOo}{{\OO_0}}
\newcommand{\OOp}{{\OO_{\rm p}}}

\newcommand{\cis}{\LC_{\rm i/s}}
\newcommand{\cqi}{\LC_{\rm qi}}
\newcommand{\crd}{\LC_{\rm rd}}
\newcommand{\cb}{\LC_{\rm b}}
\newcommand{\ci}{\LC_{\rm i}}
\newcommand{\ce}{\LC_{\rm e}}

\newtheorem{theorem}{Theorem}
\newtheorem{lemma}[theorem]{Lemma}
\newtheorem{assumption}[theorem]{Assumption}

\numberwithin{theorem}{section}
\numberwithin{equation}{section}

\newcommand{\bthe}{\boldsymbol \theta}
\newcommand{\bphi}{\boldsymbol \phi}

\newcommand{\tbv}{\widetilde{\bv}}
\newcommand{\tbw}{\widetilde{\bw}}

\newcommand{\tK}{\widetilde K}
\newcommand{\tF}{\widetilde F}

\newcommand{\essinf}[1]{\underset{\substack{#1}}{\operatorname{ess} \operatorname{inf}}\;}
\newcommand{\esssup}[1]{\underset{\substack{#1}}{\operatorname{ess} \operatorname{sup}}\;}

\newcommand{\osc}{\operatorname{osc}}

\newcommand{\omegaP}{\omega_{\rm P}}

\newcommand{\kE}{k_{\rm E}}
\newcommand{\kJ}{k_{\rm J}}
\newcommand{\kP}{k_{\rm P}}
\newcommand{\cE}{c_{\rm E}}
\newcommand{\cJ}{c_{\rm J}}
\newcommand{\cP}{c_{\rm P}}

\newcommand{\kEK}{k_{{\rm E},K}}
\newcommand{\kJK}{k_{{\rm J},K}}
\newcommand{\kPK}{k_{{\rm P},K}}
\newcommand{\cEK}{c_{{\rm E},K}}
\newcommand{\cJK}{c_{{\rm J},K}}

\newcommand{\scurl}{\operatorname{curl}}
\newcommand{\vcurl}{\boldsymbol{\operatorname{curl}}}

\title[A posteriori error estimates for non-local Maxwell-Drude equations]%
{A posteriori error estimates for finite element discretizations of time-harmonic
Maxwell's equations coupled with a non-local hydrodynamic Drude model}

\author{T. Chaumont-Frelet$^{\dagger,\ddagger}$}
\author{S. Lanteri$^{\dagger,\ddagger}$}
\author{P. Vega$^\dagger$}

\address{\vspace{-.5cm}}
\address{\noindent \tiny \textup{$^\dagger$Inria, 2004 Route des Lucioles, 06902 Valbonne, France}}
\address{\noindent \tiny \textup{$^\ddagger$Laboratoire J.A. Dieudonn\'e, Parc Valrose, 28 Avenue Valrose, 06108 Nice Cedex 02, 06000 Nice, France}}

\begin{document}	
	
\maketitle	

\begin{abstract}
We consider finite element discretizations of Maxwell's equations coupled
with a non-local hydrodynamic Drude model that accurately accounts for
electron motions in metallic nanostructures. Specifically, we focus on
{\it a posteriori} error estimation and mesh adaptivity, which is of particular
interest since the electromagnetic field usually exhibits strongly localized features
near the interface between metals and their surrounding media. We propose a novel
residual-based error estimator that is shown to be reliable and efficient. We also
present a set of numerical examples where the estimator drives a mesh adaptive process.
These examples highlight the quality of the proposed estimator, and the potential
computational savings offered by mesh adaptivity.

\vspace{.5cm}
\noindent
{\sc Key words.}
A posteriori error estimates;
Finite element methods;
Maxwell's equations;
Non-local hydrodynamic Drude model;
Plasmonics
\end{abstract}

\section{Introduction}

The interaction of light with metallic nanostructures gives rise to so-called
plasmonic waves that are due to collective oscillations of conduction
band electrons in the metal, and typically concentrate at the interface between
the nanostructure and the surrounding medium. These unusual properties allow an
extraordinary level of light manipulation at the nanoscale \cite{maier_2007a},
with applications in waveguiding \cite{oulton_sorger_genov_pile_zhang_2008},
lasing \cite{smalley_vallini_gu_fainman_2016}, near-field scanning microscopy
\cite{novotny_vanhulst_2011}, ultrasensitive sensing
\cite{stewart_anderton_thompson_maria_gray_rogers_nuzzo_2008} and
energy harvesting \cite{brongersma_2016}.

Electromagnetic fields penetrate in noble metals up to 25~nm whatever
the wavelength. Small metallic nanostructures actually contains a plasma, whose
electromagnetic response is in opposition to the incoming field, generating plasmonic
waves. While this effect is negligible when considering large structures, metals cannot
be considered to be perfectly conducting at the nanoscale, and valence electrons have to
be modeled as a gas \cite{maier_2007a}, leading to dispersive material laws.

In this work, we focus on the time-harmonic setting where the electromagnetic field
oscillates in time at a prescribed frequency $\omega > 0$. In this context, the
Drude model \cite{drude_1900} is a fairly simple yet efficient oscillator
model for free electrons in metals. Standard Maxwell's equations are employed
in the metal to describe the propagation of the electric field $\BE$,
\begin{equation}
\label{eq_maxwell_drude_local}
-\omega^2 \varepsilon_{\rm d} \BE + \curl \left (\mu_0^{-1} \curl \BE\right )
=
\bo
\end{equation}
but the permittivity
\begin{equation*}
\varepsilon_{\rm d}
\eq
\left (1 - \frac{\omegaP^2}{i\gamma\omega+\omega^2}\right )
\varepsilon_0
\end{equation*}
becomes a complex-valued function of the frequency, with a negative real part at optical
frequencies. Above, $\varepsilon_0$ and $\mu_0$ respectively denote the vacuum electric
permittivity and magnetic permeability. $\omegaP$ and $\gamma$ are the so-called
``plasma'' and ``collision'' frequencies of the metal under consideration.
Although the Drude model performs well in most cases, it becomes inaccurate when the
size of the considered nanostructure decreases beyond approximately 10~nm.
Then, a possible extension is the so-called non-local hydrodynamic Drude (NHD) model
\cite{raza_bozhevolnyi_wubs_mortensen_2015}, where the electron gas is treated as a fluid.
Compared to the ``local'' Drude model for free electrons, this hydrodynamic approach
accounts for the Fermi velocity $\vartheta_{\rm F}$ of the electrons via an additional
parameter $\beta^2 \eq (3/5)\vartheta_{\rm F}^2$, namely
\begin{equation}
\label{eq_maxwell_drude}
\left \{
\begin{array}{rcl}
-\omega^2 \varepsilon_0 \BE
+
\curl \left (\mu_0^{-1} \curl \BE\right )
+
i\omega \BJ
&=&
\bo,
\\
-\omega^2 \BJ - i\omega\gamma\BJ - \grad \left (\beta^2 \div \BJ\right )
-i\omega \omega_{\rm P}^2 \varepsilon_0 \BE
&=&
\bo,
\end{array}
\right .
\end{equation}
where the motion of the electrons is now explicitly modeled through
the velocity field $\BJ$. Notice that setting $\beta \eq 0$,
\eqref{eq_maxwell_drude} reduces to \eqref{eq_maxwell_drude_local}.
%
%
Thanks to its relatively
simple form and the successful interpretation of observable non-local
effects \cite{duan_fernandezdominguez_bosman_maier_yang_2012}, the NHD model
has become a quite popular approach in the study of optical properties
of metallic nanostructures.

The above considerations have naturally led to an increasing interest
for efficient numerical discretizations of Maxwell's equations coupled
with NHD model \eqref{eq_maxwell_drude} in metallic nanostructures.
Several approaches have been considered, including boundary integral
equations \cite{zheng_kupresak_mittra_vandenbosch_2018}, discontinuous
Galerkin schemes
\cite{li_lanteri_mortensen_wubs_2017a,vidalcodina_nguyen_oh_peraire_2018a}
and finite element methods
\cite{hiremath_zschiedrich_schmidt_2012a,ma_zhang_zou_2019,toscano_raza_jauho_mortensen_wubs_2012}.
Here,  we focus  on  finite element  discretizations,  which have  the
advantage to easily handle heterogeneous media as compared to integral
equations,  while   being  simpler  to  implement   and  analyze  than
discontinuous Galerkin  schemes. The  ability to work  on unstructured
meshes not only permits to deal with arbitrary geometries, but it also
allows for  local mesh  refinements.  Such local  refinements increase
the  accuracy in  those areas  where  the solution  exhibit a  complex
behavior at  a reduced cost,  and seem  of particular interest  in the
context  of nanoplasmonics,  since  plasmons are  in general  strongly
localized. Here,  we thus focus on  the design and analysis  of {\it a
  posteriori}  error  estimators,  and  their ability  to  drive  mesh
adaptive algorithms \cite{ainsworth_oden_2000a,verfurth_1994}.

Our main contributions are threefold. First, we propose a novel {\it a posteriori} error
estimator for finite element discretizations of the Maxwell-NHD system in general
three-dimensional configurations.  Our estimator is of ``residual'' type, and builds upon
previous constructions for standard Maxwell's equations
\cite{beck_hiptmair_hoppe_wohlmuth_2000a,chaumontfrelet_vega_2020a,nicaise_creuse_2003a}
with suitable modifications to handle the NHD model. Our second key contribution is a detailed
mathematical analysis of the estimator, where we show that it is both reliable and efficient
in a suitable energy norm. Finally, we numerically evaluate the ability of the estimator to
drive adaptive processes, and quantify the computational savings as compared to uniform
meshes. To this end, we consider three two-dimensional examples that are representative of
typical nanoplasmonic applications. In each case, the use of adaptivity leads to a drastic
reduction of the number of required degrees of freedom to achieve any given accuracy. These
preliminary results are very promising in view of more realistic three-dimensional applications.

To the best of our knowledge, most existing studies on the NHD model focus on the
development of numerical methods or the analysis of physical effects. In comparison,
the rigorous mathematical analysis is relatively recent, and {\it a priori}
error estimates have been only recently established \cite{ma_zhang_zou_2019}.
As a result, the proposed analysis appears to be entirely original.
	
The remainder of this manuscript is organized as follows.
In Section \ref{section_settings}, we present our model problem,
notations and preliminary results. Section \ref{section_estimator} introduces the
\textit{a posteriori} error estimator and establishes our main theoretical results.
We provide numerical examples in Section \ref{section_numerics} and draw our conclusions
in Section \ref{section_conclusions}.

\section{Settings}
\label{section_settings}

\subsection{Maxwell-NHD equations}

In this work, $\OO \subset \mathbb R^3$ is a polyhedral Lipschitz domain,
and $\OOm \subset \subset \OO$ is a Lipschitz polyhedral subset. $\OOm$ can be
multi-connected, but the boundaries of $\OO$ and $\OOm$ are not allowed
to touch each other. Whenever convenient, we will implicitly extend scalar-valued
(resp. vector-valued) functions defined in $\OOm$ by $0$ (resp. $\bo$) in $\OO \setminus \OOm$.

For the sake of simplicity, we consider a generalization of \eqref{eq_maxwell_drude}
with more general coefficients that we now describe. We assume that $\OO$ is subdivided into
a polyhedral partition $\LP$ that is conforming with $\OOm$ in the sense that any subset
$P \in \LP$ either entirely belongs to $\overline{\OO}$ or $\overline{\OOm}$.
Then $\ee,\cc: \OO \to \LL(\mathbb C^3)$, $\al: \OOm \to \LL(\mathbb C^3)$ and
$\zeta: \OOm \to \mathbb C$ are assumed to be piecewise constant onto $\LP$.

We do not require the tensor-valued functions neither to be symmetric,
nor to be positive-definite. Also, $\zeta$ is allowed to change sign.
The only assumption we require is that the problem is inf-sup stable
(see Assumption \ref{assumption_inf_sup} below) which implicitly
constrains the coefficients.

For $\bx \in \OO$, we denote by
\begin{equation*}
\varepsilon^\star(\bx)
\eq
\max_{\substack{\bu \in \mathbb C^3 \\ |\bu| = 1}}
\max_{\substack{\bv \in \mathbb C^3 \\ |\bv| = 1}}
\Re \left (\ee(\bx) \bu \cdot \overline{\bv}\right ),
\qquad
\chi^\star(\bx)
\eq
\max_{\substack{\bu \in \mathbb C^3 \\ |\bu| = 1}}
\max_{\substack{\bv \in \mathbb C^3 \\ |\bv| = 1}}
\Re \left (\cc(\bx) \bu \cdot \overline{\bv}\right )
\end{equation*}
and similarly, for $\bx \in \OOm$, we write
\begin{equation*}
\alpha^\star(\bx)
\eq
\max_{\substack{\bu \in \mathbb C^3 \\ |\bu| = 1}}
\max_{\substack{\bv \in \mathbb C^3 \\ |\bv| = 1}}
\Re \left (\al(\bx) \bu \cdot \overline{\bv}\right ),
\qquad
\zeta^\star(\bx)
\eq
|\zeta(\bx)|.
\end{equation*}
For the sake of simplicity, we implicitly extend $\alpha^\star$
and $\zeta^\star$ by zero in $\OO \setminus \overline{\OOm}$.
If $\phi\in\{\varepsilon,\chi,\alpha,\zeta\}$, we introduce the notations
$\phi_K^\star \eq \phi^\star|_K \in \mathbb R$ as well as
\begin{equation*}
\phi_{D,\min} \eq \essinf{\bx \in D} \phi^\star(\bx),
\qquad
\phi_{D,\max} \eq \esssup{\bx \in D} \phi^\star(\bx)
\end{equation*}
for any open set $D \subset \OO$, and we assume that
$\varepsilon_{\OO,\min},\mu_{\OO,\min},\alpha_{\OOm,\min},\zeta_{\OOm,\min} > 0$.

We are now ready to state our model problem. Namely, given $\BJe: \Omega \to \mathbb C^3$
and $\BKe: \OOm \to \mathbb C^3$, we seek $\BE: \Omega \to \mathbb C^3$ and
$\BJ: \OOm \to \mathbb C^3$ such that
\begin{equation}
\label{eq_maxwell_drude_strong}
\left \{
\begin{array}{rcll}
-\omega^2 \ee \BE
+
\curl \left (\cc \curl \BE\right )
+
i\omega \BJ
&=&
i\omega \BJ_{\rm e}
&
\text{ in } \OO,
\\
-\omega^2 \al \BJ  - \grad \left (\zeta \div \BJ\right )
-i\omega \BE
&=&
i\omega\BKe
&
\text{ in } \OOm.
\end{array}
\right .
\end{equation}
On the one hand, we recover \eqref{eq_maxwell_drude} when
\begin{equation*}
\ee \eq \varepsilon_0 \BI, \quad
\cc \eq \mu_0^{-1} \BI, \quad
\al \eq \frac{1}{\omega_{\rm P}^2 \eps_0}
\left(1 - \frac{1}{i\omega} \gamma\right)\BTI,
\quad
\zeta \eq \frac{\beta^2}{\omega_{\rm P}^2 \eps_0}.
\end{equation*}
On the other hand, the proposed reformulation permits to treat more general
cases in a uniform manner, without adding any mathematical complexity.
In particular, the permittivity is allowed to change sign, which enables
to take into account the local Drude model. Besides, our analysis naturally
handles anisotropic materials, and in particular, perfectly matched layers
can be employed to model unbounded propagation media \cite{monk_2003a}.

For later use, we notice that as usual in the analysis of Maxwell's equations,
there are two ``hidden'' equations in \eqref{eq_maxwell_drude_strong}, namely
\begin{equation}
\label{eq_hidden_div}
\div \left (i\omega \ee \BE + \BJ\right ) = \div \BJ_{\rm e} \text{ in } \OO
\end{equation}
and
\begin{equation}
\label{eq_hidden_curl}
\curl \left (i\omega \al \BJ - \BE \right ) = \curl \BK_{\rm e} \text{ in } \OOm.
\end{equation}

We also notice that there are three different (space-dependent) wavenumbers appearing
in the above model.
\begin{subequations}
\label{eq_wavenumbers}
As usual, the electromagnetic wavenumber is defined in $\OO$ by
\begin{equation}
\kE \eq \frac{\omega}{\cE},
\quad
\cE \eq \sqrt{\frac{\chi^\star}{\eps^\star}},
\end{equation}
and in addition, we introduce
\begin{equation}
\kJ \eq \frac{\omega}{\cJ},
\quad
\kP \eq \frac{\omegaP}{\cP},
\quad
\cJ \eq \sqrt{\frac{\zeta^\star}{\alpha^\star}},
\quad
\cP \eq \omegaP \sqrt{\zeta^\star\eps^\star},
\end{equation}
in $\OOm$. We add, for all the notations in \eqref{eq_wavenumbers},
a second subscript $K$ for the (constant) restrictions to $K \in \CT_h$.
\end{subequations}


\subsection{Functional spaces}

If $D \subset \Omega$ is an open set, $L^2(D)$
denotes the space of complex-valued square-integrable functions
defined on $D$, and $\BL^2(D) \eq \left (L^2(D)\right )^3$.
The notations $\|\cdot\|_D$ and $(\cdot,\cdot)_D$ stand for
the usual norm and inner-product of $L^2(D)$ and $\BL^2(D)$.
For $\phi \in \{\varepsilon,\chi,\alpha,\zeta\}$,
we introduce the (equivalent) norms on $L^2(D)$ and $\BL^2(D)$ defined by
\begin{equation*}
\|w\|_{\phi,D}^2\eq\int_D\phi^\star|w|^2,
\qquad
\|\bw\|_{\phi,D}^2\eq\int_D\phi^\star|\bw|^2,
\end{equation*}
$w \in L^2(D)$ and $\bw \in \BL^2(D)$.

$H^1(D)$ is the usual first-order Sobolev space of functions
$v \in L^2(D)$ such that $\grad v \in \BL^2(D)$. If $\Gamma \subset \partial \Omega$
is a relatively open set, $H^1_\Gamma(D)$ stands for the space of functions $v \in H^1(D)$
such that $v|_\Gamma = 0$. For vector-valued functions, we also introduce
$\BH^1(D) \eq \left (H^1(D)\right )^3$ and $\BH^1_\Gamma(D) \eq \left (H^1_\Gamma(D)\right )^3$.

We will also need the vector Sobolev spaces
\begin{align*}
\BH(\ccurl,D)
&\eq
\left \{
\bv \in \BL^2(D) \; | \; \curl \bv \in \BL^2(D)
\right \},
\\
\BH(\ddiv,D)
&\eq
\left \{
\bv \in \BL^2(D) \; | \; \div \bv \in L^2(D)
\right \},
\end{align*}
and there subspaces $\BH_0(\ccurl,D)$ and $\BH_0(\ddiv,D)$ that are defined,
as usual, as the closure of smooth compactly supported functions.

The aforementioned functional spaces are widely documented in the literature,
and we refer the reader to \cite{adams_fournier_2003a,girault_raviart_1986a}
for a precise description.

We finally introduce the ``energy'' space
\begin{equation*}
\mathbb V \eq \BH_0(\ccurl,\OO) \times \BH_0(\ddiv,\OOm),
\end{equation*}
that we equip with the norm
\begin{equation*}
\enorm{(\be,\bj)}^2
\eq
\omega^2 \|\be\|_{\varepsilon,\OO}^2 + \|\curl \be\|_{\chi,\OO}^2
+
\omega^2 \|\bj\|_{\alpha,\OOm}^2 + \|\div \bj\|_{\zeta,\OOm}^2,
\quad
(\be,\bj) \in \mathbb V.
\end{equation*}
If $D \subset \OO$ is an open set, we will also use the local version
\begin{equation*}
\enorm{(\be,\bj)}_D^2
\eq
\omega^2 \|\be\|_{\varepsilon,D}^2 + \|\curl \be\|_{\chi,D}^2
+
\omega^2 \|\bj\|_{\alpha,D \cap \OOm}^2 + \|\div \bj\|_{\zeta,D \cap \OOm}^2,
\quad
(\be,\bj) \in \mathbb V.
\end{equation*}

\subsection{Well-posedness}

We denote by $b: \mathbb V \times \mathbb V \to \mathbb C$
the sesquilinear form naturally associated with \eqref{eq_maxwell_drude_strong}
after integration by parts. It is defined by
\begin{align*}
b((\be,\bj),(\bv,\bw))
\eq&
-\omega^2 (\ee\be,\bv)_\OO + (\cc\curl \be,\curl \bv)_\OO
+ i\omega (\bj,\bv)_\OOm
\\
&-\omega^2 (\al\bj,\bw)_\OOm
+
(\zeta \div \bj,\div \bw)_{\OOm} - i\omega (\be,\bw)_{\OOm}
\end{align*}
for all $(\be,\bj),(\bv,\bw) \in \mathbb V$. Then, a weak
formulation of \eqref{eq_maxwell_drude_strong} consists in
finding $(\BE,\BJ) \in \mathbb V$ such that
\begin{equation}
\label{eq_maxwell_drude_weak}
b((\BE,\BJ),(\bv,\bw)) = i\omega (\BJe,\bv) +i\omega (\BKe,\bw)
\qquad
\forall (\bv,\bw) \in \mathbb V.
\end{equation}


In the remaining of this work, we require that the sesquilinear form $b$ is inf-sup stable,
which implies well-posedness of \eqref{eq_maxwell_drude_weak}. Specifically, we make the
following assumption.

\begin{assumption}[Well-posedness]
\label{assumption_inf_sup}
There exists a constant $\cis$ such that
\begin{equation}
\label{eq_inf_sup}
\inf_{(\be,\bj) \in \mathbb V \setminus \{0\}}
\sup_{(\bv,\bw) \in \mathbb V \setminus \{0\}}
\frac{\Re b((\be,\bj),(\bv,\bw))}{\enorm{(\be,\bj)}\enorm{(\bv,\bw)}}
\geq
\cis
>
0.
\end{equation}
\end{assumption}

\subsection{Mesh}

We consider a mesh $\CT_h$ of $\OO$ made of tetrahedral elements $K$.
$\CT_h$ is conforming in the sense of \cite{ciarlet_2002a}, which means that the
intersection $\overline{K}_- \cap \overline{K_+}$ of two distinct
elements $K_\pm \in \CT_h$ is either empty, or a single vertex, edge
or face of both elements. We further assume that $\CT_h$ is conforming with $\LP$,
in the sense for each $K \in \CT_h$, there exists $P \in \LP$ such that
$\overline{K} \subset \overline{P}$. We denote by $\CT_{{\rm m},h}$ the restriction
of $\CT_h$ to $\OOm$, i.e., the set of those $K \in \CT_h$ such that
$\overline{K} \subset \overline{\OOm}$.

Following \cite{ciarlet_2002a}, we employ the notations
\begin{equation*}
h_K \eq \max_{\bx,\by \in \overline{K}} |\bx-\by|,
\qquad
\rho_K \eq
\max
\left \{
r > 0 \; | \; \exists \bx \in \overline{K}; \; B(\bx,r) \subset \overline{K}
\right \}
\end{equation*}
for the diameter and inscribed sphere radius of the element $K \in \CT_h$.
$\kappa_K \eq h_K/\rho_K$ is then called the shape-regularity parameter of $K$,
and $\kappa \eq \max_{K \in \CT_h} \kappa_K$ is the shape-regularity parameter
of $\CT_h$.

We introduce, for $K \in \CT_h$ and $F \in \CF_h$, the sets
\begin{align*}
\CT_{K,h} \eq \{K' \in \CT_h \; | \; \overline{K} \cap \overline{K}' \neq \emptyset\},
\qquad
\CT_{F,h} \eq \{K'\in\CT_h \; | \;  F \subset \partial K'\},
\end{align*}
and the associated open domains
\begin{align*}
\tK \eq \operatorname{Int} \left(\bigcup_{K'\in\CT_{K,h}} \overline{K'} \right),
\qquad
\tF \eq \operatorname{Int} \left(\bigcup_{K'\in\CT_{F,h}} \overline{K'} \right).
\end{align*}
When $K \in \CT_{{\rm m},h}$ and $F \in \CF_{{\rm m},h}$,
we will also use the submeshes
$\CT_{{\rm m},K,h} \eq \CT_{K,h} \cap \CT_{{\rm m},h}$ and
$\CT_{{\rm m},F,h} \eq \CT_{F,h} \cap \CT_{{\rm m},h}$,
and the associated open domains $\tK_{\rm m}$ and $\tF_{\rm m}$.

For $\CT \subset \CT_h$ and $\CF \subset \CF_h$, we write
\begin{align*}
(\cdot,\cdot)_{\CT}
\eq
\sum_{K\in\CT}(\cdot,\cdot)_K,
\qquad
\langle\cdot,\cdot\rangle_{\partial\CT}
\eq
\sum_{K\in\CT}\langle\cdot,\cdot\rangle_{\partial K},
\qquad
\langle\cdot,\cdot\rangle_{\CF}
\eq
\sum_{F\in\CF}\langle\cdot,\cdot\rangle_F
\end{align*}
and
\begin{equation*}
\BH^1(\CT)
\eq
\left \{
\bv \in \BL^2(U) \; | \; \bv|_K \in \BH^1(K) \quad \forall K \in \CT
\right \},
\end{equation*}
with $U \eq \operatorname{Int} (\cup_{K \in \CT} \overline{K})$.


\subsection{Finite element spaces}

The usual Lagrange and N\'ed\'elec spaces on $\CT_h$ read
\begin{equation*}
V_h \eq \CP_{p+1}(\CT_h) \cap H^1_0(\Omega),
\qquad
\BW_h \eq \BCN_p(\CT_h) \cap \BH_0(\ccurl,\Omega),
\end{equation*}
and we have $\grad V_h \subset \BW_h$. We shall also need the N\'ed\'elec and Raviart-Thomas
finite element spaces in the metallic part of the domain, namely
\begin{equation*}
\BW_{{\rm m},h} \eq \BCN_p(\CT_{{\rm m},h}) \cap \BH_0(\ccurl,\OOm),
\qquad
\BX_{{\rm m},h} \eq {\BCR\BCT}_p(\CT_{{\rm m},h}) \cap \BH_0(\ddiv,\OOm).
\end{equation*}
We have $\curl \BW_{{\rm m},h} \subset \BW_{{\rm m},h}$. In addition, if extension
by zero is implicitly assume, then $\BW_{{\rm m},h} \subset \BW_h$. We refer the
reader to \cite{monk_2003a} for a detailed description of these finite element
spaces.

\subsection{Quasi-interpolation operators}

\begin{subequations}
Classically, our analysis will rely on ``quasi-interpolation'' operators
\cite{ern_guermond_2017a}. Specifically, there exist four operators
$\CP_h: H^1_0(\OO) \to V_h$, $\CQ_h: \BH_0(\ccurl,\OO) \to \BW_h$,
$\CQ_{{\rm m},h}: \BH_0(\ccurl,\OOm) \to \BW_{{\rm m},h}$ and
$\CR_{{\rm m},h}: \BH_0(\ddiv,\OOm) \to \BX_{{\rm m},h}$
and a constant $\cqi$ that only depends on the shape-regularity
parameter $\kappa$ such that
\begin{align}
\label{eq_qi_H1}
h_K^{-1} \|q-\CP_hq \|_K
+
h_K^{-1/2} \|q-\CP_h q\|_{\partial K}
&\leq
\cqi
\|\grad q\|_{\tK},
\\
\label{eq_qi_Hc}
h_K^{-1} \|\bv-\CQ_h\bv \|_K
+
h_K^{-1/2} \|(\bv-\CQ_h \bv) \times \bn\|_{\partial K}
&\leq
\cqi
\|\grad_h \bv\|_{\tK}
\end{align}
for all $q \in H^1_0(\Omega)$, $\bv \in \BH^1(\CT_h) \cap \BH_0(\ccurl,\OO)$
and $K \in \CT_h$, as well as
\begin{align}
\label{eq_qi_Hcm}
h_K^{-1} \|\bv-\CQ_{{\rm m},h}\bv \|_K
+
h_K^{-1/2} \|(\bv-\CQ_{{\rm m},h} \bv) \times \bn\|_{\partial K}
&\leq
\cqi
\|\grad_h \bv\|_{\tK_{\rm m}},
\\
\label{eq_qi_Hdm}
h_K^{-1} \|\bw-\CR_{{\rm m},h}\bw \|_K
+
h_K^{-1/2} \|(\bw-\CR_{{\rm m},h} \bw) \cdot \bn\|_{\partial K}
&\leq
\cqi
\|\grad_h \bw\|_{\tK_{\rm m}}
\end{align}
for all $\bv \in \BH^1(\CT_{{\rm m},h}) \cap \BH_0(\ccurl,\OOm)$,
$\bw \in \BH^1(\CT_{{\rm m},h}) \cap \BH_0(\ddiv,\OOm)$,
and $K \in \CT_{{\rm m},h}$.
\end{subequations}

\subsection{Bubble functions}
\label{bubbles}
	
\begin{subequations}
Classically, we will use ``bubble'' functions to localize our error analysis.
We refer the reader to \cite{verfurth_1994} for a detailed presentation and only
state the essential result we need. Given an element $K \in \CT_h$ and a face $F \in \CF_h$,
we denote by $b_K \in H_0^1(K)$ and $b_F \in H_0^1(\tF)$ the element and face bubble functions
supported in $K$ and $\tF$, respectively. The estimates
\begin{equation}
\label{eq_norm_bubble}
\|w\|_K \leq \cb\|b_K^{1/2} w\|_K, \qquad \|v\|_F \leq \cb \|b_F^{1/2} v\|_F
\end{equation}
hold for all $w \in \CP_{p+1}(K)$ and $v \in \CP_{p+1}(F)$, where $\cb>0$
is a constant depending on the polynomial degree $p$ and the
shape regularity parameter $\kappa$. Standard inverse inequalities let us conclude that
\begin{equation}
\label{eq_inv_bubble}
\|\grad(w b_K)\|_K \leq \ci h_K^{-1}\|w\|_K
\qquad
\forall w \in \CP_{k+1}(K),
\end{equation}
where again, $\ci$ only depends on $\kappa$ and $p$.
We further consider an extension operator $\LL_F: \CP_{k+1}(F) \to \CP_{p+1}(\tF)$
such that $\LL_F(v)|_F=v$ and
\begin{equation}
\label{eq_ext_bubble}
\|\LL_F(v)\|_{\tF} + h_F \|\grad(\LL_F(v))\|_{\tF}
\leq
\ce h_F^{1/2}\|v\|_F, \qquad v \in \CP_{k+1}(F),
\end{equation}
where $\ce$ depends on $\beta$ and $p$.
\end{subequations}
The same results hold true for vector-valued function,
as can be seen by applying the scalar estimates componentwise.
	
\subsection{Data oscillation}
\label{section_oscillation}

Our efficiency estimates include a data-oscillation term that we define
in this section. We first define a ``projected source term'' $\pi_h \BJe \in \BCP_{p+1}(\CT_h)$,
that is defined for each $K \in \CT_h$ as the unique element in $\BCP_{p+1}(K)$ such that
\begin{equation*}
\frac{\omega^2 h_K^2}{\cEK^2}(\pi_h\BJe,\bv_h)_K
+
h_K^2(\div(\pi_h\BJe),\div\bv_h)_K
=
\frac{\omega^{2} h_K^{2}}{\cEK^{2}}(\BJe,\bv_h)_K
+
h_K^2(\div\BJe,\div\bv_h)_K
\end{equation*}
for all $\bv_h \in \BCP_{k+1}(K)$. Analogously, we define $\varrho_h \BKe\in \BCP_{p+1}(\CT_h)$,
for each $ K\in\CT_h $, as the unique element in $ \BCP_{p+1}(K) $ such that
\begin{equation*}
\frac{\omega^2 h_K^2}{\cJK^2}(\varrho_h\BKe,\bw_h)_K
+
h_K^2(\curl(\varrho_h\BKe),\curl\bw_h)_K
=
\frac{\omega^{2} h_K^{2}}{\cJK^{2}}(\BKe,\bw_h)_K
+
h_K^2(\curl\BKe,\curl\bw_h)_K
\end{equation*}
for all $\bw_h \in \BCP_{k+1}(K)$.
Then, we may introduce the data oscillation term
\begin{align*}
\osc_K^2
&\eq
\frac{1}{\eps_K^\star}
\left (
\frac{\omega^2 h_K^2}{\cEK^2}\|\BJe-\pi_h\BJe\|_K^2 + h_K^2\|\div(\BJe-\pi_h\BJe)\|_K^2
\right )
\\
&\;+
\frac{1}{\alpha_K^\star}
\left (
\frac{\omega^2 h_K^2}{\cJK^2}\|\BKe-\varrho_h\BKe\|_K^2 + h_K^2\|\curl(\BKe-\varrho_h\BKe)\|_K^2
\right )
\end{align*}
for all $K \in \CT_h$, and 
\begin{equation*}
\osc_{\tK}^2 \eq \sum_{K \in \CT_h^K} \osc_K^2.
\end{equation*}

Notice that whenever the right-hand side is smooth, namely
$\BJe \in \BH^{p+1}(\CT_h)$, $\div \BJe \in H^{p+1}(\CT_h)$,
$\BKe \in \BH^{p+1}(\CT_h)$ and $\curl \BKe \in \BH^{p+1}(\CT_h)$,
we have $\osc_K = O(h^{p+2})$ for all $K \in \CT_h$.

\subsection{Regular decomposition}

For all $(\bv,\bw) \in \mathbb V$, there exist
$q \in H^1_0(\OO)$,
$\bthe \in \BH^1(\CT_{{\rm m},h}) \cap \BH_0(\ccurl,\OOm)$,
$\tbv \in \BH^1(\CT_h) \cap \BH_0(\ccurl,\OO)$
and
$\tbw \in \BH^1(\CT_{{\rm m},h}) \cap \BH_0(\ddiv,\OOm)$
such that
\begin{subequations}
\label{eq_regular_decomposition}
\begin{equation}
\label{eq_regular_decomposition_identity}
(\bv,\bw) = (\grad p + \tbv,\curl \bthe + \tbw)
\end{equation}
and
\begin{equation}
\label{eq_regular_decomposition_estimate}
\omega \|\grad q\|_{\varepsilon,\OO} + \omega \|\grad_h \bthe\|_{\alpha,\OOm}
+
\|\grad_h \tbv\|_{\chi,\OO} + \|\grad_h \tbw\|_{\zeta,\OOm}
\leq
\crd
\enorm{(\bv,\bw)},
\end{equation}
where $\crd$ is a constant that only depends on $\LP$, the ratio between
the minimum and maximum value of the coefficients. We refer the reader to
Theorems 2 and 6 of \cite{hiptmair_pechstein_2019a}, as well as
\cite{chaumontfrelet_vega_2020a}, Appendix B.
\end{subequations}

\subsection{Inequalities with hidden constants}

To simplify the remaining of the exposition, if $A,B \geq 0$
are real numbers, we employ the notation $A \lesssim B$ if
there exists a constant $C$ that only depends on
$\cis$, $\cqi$, $\cb$, $\ce$, $\ci$, $\crd$ and the material
contrasts such that $A \leq CB$. In particular, $C$ may depend
on the geometry of the domain and the material coefficients,
the mesh shape-regularity $\kappa$ and the polynomial degree $p$,
but \emph{not} on the mesh size $h$.

\section{A posteriori error estimates}
\label{section_estimator}

\subsection{Numerical solution}

We are interested in finite element approximations to \eqref{eq_maxwell_drude_weak}.
Specifically, we introduce the (conforming) discretization space
$\mathbb V_h \eq \BW_h \times \BX_{{\rm m},h}$ and consider an element
$(\BE_h,\BJ_h) \in \mathbb V_h$ such that
\begin{equation}
\label{eq_galerkin_orthogonality}
b((\BE_h,\BJ_h),(\bv_h,\bw_h)) = i\omega (\BJe,\bv_h)+i\omega (\BKe,\bw_h)
\qquad
\forall (\bv_h,\bw_h) \in \mathbb V_h.
\end{equation}

\subsection{A posteriori error estimator}

We devise a residual-based error estimator. It is based on four terms.
The first two are motivated by the two equations of
\eqref{eq_maxwell_drude_strong} and read
\begin{multline*}
\eta_{\ccurl,\ccurl,K}
\eq
\frac{h_K}{\sqrt{\chi_K^\star}}
\|
-\omega^2 \ee \BE_h + \curl \left (\cc \curl \BE_h\right ) + i\omega \BJ_h - i\omega \BJ_{\rm e}
\|_K
\\
+
\frac{h_K^{1/2}}{\sqrt{\chi_K^\star}}
\|\jmp{\cc \curl \BE_h} \times \bn\|_{\partial K \setminus \partial \OO}
\end{multline*}
for all $K \in \CT_h$ and
\begin{multline*}
\eta_{\ggrad,\ddiv,K}
\eq
\frac{h_K}{\sqrt{\zeta_K^\star}}
\|
-\omega^2 \al \BJ_h
-\grad \left (\zeta \div \BJ_h \right )
-i\omega \BE_h- i\omega \BK_{\rm e}
\|_K
\\
+
\frac{h_K^{1/2}}{\sqrt{\zeta_K^\star}}
\|\jmp{\zeta \div \BJ_h}\|_{\partial K \setminus \partial \OOm}
\end{multline*}
for $K \in \CT_{{\rm m},h}$. As usual in the context of Maxwell's equations
\cite{beck_hiptmair_hoppe_wohlmuth_2000a,nicaise_creuse_2003a},
these two terms are insufficient, and we also need to consider the residual terms
associated with ``hidden'' equations \eqref{eq_hidden_div} and \eqref{eq_hidden_curl}.
Hence, we introduce
\begin{equation*}
\eta_{\ddiv,K}
\eq
\frac{h_K}{\sqrt{\varepsilon_K^\star}}
\|\div(i\omega\ee\BE_h  + \BJ_h-\BJ_{\rm e})\|_K
+
\frac{\omega h_K^{1/2}}{\sqrt{\varepsilon_K^\star}}
\|\jmp{\ee \BE_h} \cdot \bn\|_{\partial K \setminus \partial \OO}
\end{equation*}
if $K \in \CT_h$ and
\begin{equation*}
\eta_{\ccurl,K}
\eq
\frac{h_K}{\sqrt{\alpha_K^\star}}
\|\curl(i\omega\al\BJ_h - \BE_h-\BK_{\rm e})\|_K
+
\frac{\omega h_K^{1/2}}{\sqrt{\alpha_K^\star}}
\|\jmp{\al \BJ_h} \times \bn\|_{\partial K \setminus \partial \OOm}
\end{equation*}
for all $K \in \CT_{{\rm m},h}$. We then set
\begin{equation*}
\eta_K
\eq
\eta_{\ccurl,\ccurl,K} + \eta_{\ggrad,\ddiv,K} + \eta_{\ddiv,K} + \eta_{\ccurl,K},
\end{equation*}
for all $K \in \CT_h$ with the implicit convention that
$\eta_{\ggrad,\ddiv,K} = \eta_{\ccurl,K} = 0$ when $K \notin \CT_{{\rm m},h}$.
Finally,
\begin{equation*}
\eta \eq \left (\sum_{K \in \CT_h} \eta_K^2 \right )^{1/2}
\end{equation*}
gathers the elementwise contributions, and we define $\eta_{\ccurl,\ccurl}$,
$\eta_{\ggrad,\ddiv}$, $\eta_{\ddiv}$ and $\eta_{\ccurl}$ in a similar way.

\subsection{Reliability}

We first establish that the proposed estimator is reliable.
The key ingredient of the proof is to estimate, for an arbitrary
element $(\bv,\bw) \in \mathbb V$, the quantity
\begin{equation*}
|b((\BE-\BE_h,\BJ-\BJ_h),(\bv,\bw))|
\end{equation*}
using the estimator $\eta$. This is done in four major steps,
that are presented in Lemmas \ref{lemma_rel_grad_o}, \ref{lemma_rel_o_curl},
\ref{lemma_rel_H1_o} and \ref{lemma_rel_o_H1} below.

\begin{lemma}
\label{lemma_rel_grad_o}
We have
\begin{equation*}
|b((\BE-\BE_h,\BJ-\BJ_h),(\grad q,\bo))|
\lesssim
\eta_{\ddiv}
\omega\|\grad q\|_{\varepsilon,\Omega}
\end{equation*}
for all $q \in H^1_0(\OO)$.
\end{lemma}

\begin{proof}
We first observe that
\begin{align*}
b((\BE-\BE_h,\BJ-\BJ_h),(\grad p,\bo))
&=
-\omega^2(\ee(\BE-\BE_h),\grad p)_{\CT_h} + i\omega(\BJ-\BJ_h,\grad p)_{\CT_h}
\\
&=
i\omega
\left (
(i\omega\ee\BE+\BJ),\grad p)_{\CT_h} - (i\omega \BE_h -\BJ_h,\grad p)_{\CT_h}
\right ).
\end{align*}
Then, since $\ee \BE,\BJ,\BJ_h \in \BH_0(\ddiv,\Omega)$ and $p \in H^1_0(\Omega)$,
elementwise integration by parts reveals that
\begin{align*}
b((\BE-\BE_h,\BJ-\BJ_h),(\grad p,\bo))
&=
i\omega
\left (
-(\div(i\omega\ee\BE+\BJ),p)_{\CT_h}
+
(\div (i\omega \ee\BE_h + \BJ_h),p)_{\CT_h}
\right )
\\
&\quad-
\omega^2 \langle \ee \BE_h \cdot \bn,p \rangle_{\partial \CT_h}
\\
&=
i\omega (\div (i\omega \ee\BE_h  + \BJ_h-\BJ_e),p)_{\CT_h}
-
\omega^2 \langle\jmp{\ee\BE_h} \cdot \bn,p\rangle_{\CF_h^{\rm i}}.
\end{align*}
Upon rearranging the face sum, it follows that
\begin{equation*}
|b((\BE-\BE_h,\BJ-\BJ_h),(\grad p,\bo))|
\lesssim
\sum_{K \in \CT_h}
\eta_{\ddiv,K}
\omega\sqrt{\varepsilon_K^\star}
\left (
h_K^{-1} \|p\|_K
+
h^{-1/2}_K \|p\|_{\partial K}
\right ).
\end{equation*}
Now, let $q \in H^1_0(\Omega)$. Using Galerkin orthogonality
\eqref{eq_galerkin_orthogonality}, the inclusion $\grad V_h \subset \BW_h$,
and quasi-interpolation estimate \eqref{eq_qi_H1}, we have
\begin{align*}
|b((\BE-\BE_h,\BJ-\BJ_h),(\grad q,\bo))|
&=
|b((\BE-\BE_h,\BJ-\BJ_h),(\grad (q-\CP_h q),\bo))|
\\
&\lesssim
\sum_{K \in \CT_h}
\eta_{\ddiv,K}
\omega\sqrt{\varepsilon_K^\star}
\left (
h_K^{-1} \|q-\CP_h q\|_K
+
h^{-1/2}_K \|q-\CP_h q\|_{\partial K}
\right )
\\
&\lesssim
\sum_{K \in \CT_h} \eta_{\ddiv,K}
\omega\sqrt{\varepsilon_K^\star} \|\grad q\|_{\CT_h^K}
\lesssim
\eta_{\ddiv} \omega \|\grad q\|_{\varepsilon,\Omega},
\end{align*}
and the result follows.
\end{proof}

\begin{lemma}
\label{lemma_rel_o_curl}
We have
\begin{equation*}
|b((\BE-\BE_h,\BJ-\BJ_h),(\bo,\curl \bthe))|
\lesssim
\eta_{\ccurl}
\omega\|\grad_h \bthe\|_{\alpha,\OOm}
\end{equation*}
for all $\bthe \in \BH^1(\CT_{{\rm m},h}) \cap \BH_0(\ccurl,\OOm)$.
\end{lemma}

\begin{proof}
Let $\bphi \in \BH^1(\CT_{{\rm m},h}) \cap \BH_0(\ccurl,\OOm)$.
The first step of the proof consists in writing that
\begin{align*}
b((\BE-\BE_h,\BJ-\BJ_h),(\bo,\curl \bphi))
&=
-\omega^2 (\al(\BJ-\BJ_h),\curl \bphi)_{\OOm}
-i\omega (\BE-\BE_h,\curl \bphi)_{\OOm}
\\
&=
i\omega \left (
(i\omega \al \BJ-\BE,\curl \bphi)_{\OOm}
-
(i\omega \al \BJ_h-\BE_h,\curl \bphi)_{\OOm}
\right ).
\end{align*}
Recalling \eqref{eq_hidden_curl}, integrating by parts over each $K \in \CT_{{\rm m},h}$
reveals that
\begin{align*}
b((\BE-\BE_h,\BJ-\BJ_h),(\bo,\curl \bphi))
&=
i\omega \left ((\curl(i\omega \al \BJ- \BE),\bphi)_{\CT_{{\rm m},h}}
- (\curl(i\omega \al \BJ_h- \BE_h),\bphi)_{\CT_{{\rm m},h}}\right)\\
&
+\;\omega^2 \langle\al \BJ_h \times \bn,\bphi \rangle_{\partial \CT_{{\rm m},h}}
\\
&=
-i\omega(\curl(i\omega \al \BJ_h- \BE_h-\BKe),\bphi)_{\CT_{{\rm m},h}}
+\omega^2 \langle\al \BJ_h \times \bn,\bphi \rangle_{\partial \CT_{{\rm m},h}}
\\
&=
-i\omega
(\curlh(i\omega \al \BJ_h-\BE_h-\BKe),\bphi)_{\OOm}
+ \omega^2 \langle\jmp{\al \BJ_h} \times \bn,\bphi\rangle_{\CF_{{\rm m},h}},
\end{align*}
and
\begin{align*}
|b((\BE-\BE_h,\BJ-\BJ_h),(\bo,\curl \bphi))|
&\lesssim
\sum_{K \in \CT_{{\rm m},h}} \eta_{\ccurl,K}
\omega\sqrt{\alpha_K^\star}
\left (
h_K^{-1} \|\bphi\|_K
+
h_K^{-1/2} \|\bphi\|_{\partial K}
\right ).
\end{align*}
Let $\bthe \in \BH^1(\CT_{{\rm m},h}) \cap \BH_0(\ccurl,\OOm)$.
Then, recalling Galerkin orthogonality \eqref{eq_galerkin_orthogonality},
that $\curl \BW_{{\rm m},h} \subset \BX_{{\rm m},h}$ and \eqref{eq_qi_Hcm}, we have
\begin{align*}
|b((\BE-\BE_h,\BJ\ -\ &\BJ_h),(\bo,\curl \bthe))|=
|b((\BE-\BE_h,\BJ-\BJ_h),(\bo,\curl (\bthe-\CQ_{{\rm m},h} \bthe))|
\\
&\lesssim
\sum_{K \in \CT_{{\rm m},h}} \eta_{\ccurl,K}
\omega\sqrt{\alpha_K^\star}
\left (
h_K^{-1} \|\bthe-\CQ_{{\rm m},h} \bthe\|_K
+
h_K^{-1/2} \|(\bthe-\CQ_{{\rm m},h}\bthe)\times\bn \|_{\partial K}
\right )
\\
&\lesssim
\sum_{K \in \CT_{{\rm m},h}} \eta_{\ccurl,K} \omega
\sqrt{\alpha_K^\star} \|\grad \bthe\|_{\CT_{{\rm m},h}^K}
\lesssim
\eta_{\ccurl} \omega\|\grad \bthe\|_{\alpha,\CT_{{\rm m},h}}.
\end{align*}
\end{proof}

\begin{lemma}
\label{lemma_rel_H1_o}
We have
\begin{equation*}
|b((\BE-\BE_h,\BJ-\BJ_h),(\tbv,\bo))|
\lesssim
\eta_{\ccurl,\ccurl}\|\grad_h \tbv\|_{\chi,\Omega}
\end{equation*}
for all $\tbv \in \BH^1(\CT_h) \cap \BH_0(\ccurl,\OO)$.
\end{lemma}

\begin{proof}
Let $\bphi \in \BH^1(\CT_h) \cap \BH_0(\ccurl,\Omega)$. We have
\begin{equation*}
b((\BE-\BE_h,\BJ-\BJ_h),(\bphi,\bo))
=
i\omega (\BJ_e,\bphi) - b((\BE_h,\BJ_h),(\bphi,\bo)),
\end{equation*}
and
\begin{align*}
b((\BE_h,\BJ_h),(\bphi,\bo))
&=
-\omega^2(\ee\BE_h,\bphi)_{\CT_h}
+
(\cc \curl \BE_h,\curl \bphi)_{\CT_h}
+
i\omega(\BJ,\bphi)_{\CT_h}
\\
&=
(-\omega^2\ee\BE_h + \curl \left (\cc \curl \BE_h\right ) + i\omega \BJ_h,\bphi)_{\CT_h}
+
\langle \left (\cc \curl \BE_h\right ) \times \bn,\bphi \rangle_{\partial \CT_h}.
\end{align*}
It follows that
\begin{align*}
&b((\BE-\BE_h,\BJ-\BJ_h),(\bphi,\bo))
\\
&=
-(-\omega^2\ee\BE_h + \curlh \left (\cc \curl \BE_h\right ) + i\omega \BJ_h - i\omega \BJ_{\rm e},\bphi)_{\OO}
+
\langle \jmp{\cc \curl \BE_h} \times \bn,\bphi \rangle_{\CF_h},
\end{align*}
and
\begin{align*}
|b((\BE-\BE_h,\BJ-\BJ_h),(\bphi,\bo))|
\lesssim
\sum_{K \in \CT_h} \eta_{\ccurl,\ccurl,K} \sqrt{\chi_K^\star} \left (
h_K^{-1} \|\bphi\|_K + h_K^{-1/2} \|\bphi \times \bn\|_{\partial K}
\right ).
\end{align*}
Then, we consider $\tbv \in \BH^1(\CT_h) \cap \BH_0(\ccurl,\OO)$ and conclude the proof
thanks to Galerkin orthogonality \eqref{eq_galerkin_orthogonality} and estimate \eqref{eq_qi_Hc},
since
\begin{align*}
|b((\BE-\BE_h,\BJ\ -\ &\BJ_h),(\tbv,\bo))|
=|b((\BE-\BE_h,\BJ-\BJ_h),(\tbv-\CQ_h\tbv,\bo))|
\\
&\lesssim
\sum_{K \in \CT_h} \eta_{\ccurl,\ccurl,K}\sqrt{\chi^\star_K}
\left (
h_K^{-1} \|\tbv-\CQ_h\tbv\|_K + h_K^{-1/2} \|(\tbv-\CQ_h\tbv) \times \bn\|_{\partial K}
\right )
\\
&\lesssim
\sum_{K \in \CT_h} \eta_{\ccurl,\ccurl,K}\sqrt{\chi_K^\star} \|\grad_h \tbv\|_{\tK}
\lesssim
\eta_{\ccurl,\ccurl}\|\grad_h \tbv\|_{\chi,\OO}.
\end{align*}
\end{proof}

\begin{lemma}
\label{lemma_rel_o_H1}
We have
\begin{equation*}
|b((\BE-\BE_h,\BJ-\BJ_h),(\bo,\tbw))|
\lesssim
\eta_{\ggrad,\ddiv} \|\grad \tbw\|_{\zeta,\CT_{{\rm m},h}}
\end{equation*}
for all $\tbw \in \BH^1(\CT_{{\rm m},h}) \cap \BH_0(\ddiv,\OOm)$.
\end{lemma}

\begin{proof}
Let $\bphi \in \BH^1(\CT_{{\rm m},h}) \cap \BH_0(\ddiv,\OOm)$.
Then, we have
\begin{align*}
b((\BE-\BE_h,\BJ-\BJ_h),(\bo,\bphi))
&=
i\omega(\BKe,\bphi)_{\CT_{{\rm m},h}}+\omega^2 (\al\BJ_h,\bphi)_{\CT_{{\rm m},h}}
-
(\zeta \div \BJ_h,\div \bphi)_{\CT_{{\rm m},h}} + i\omega (\BE_h,\bphi)_{\CT_{{\rm m},h}}
\\
&=
-\left (
-\omega^2 \al \BJ_h - \grad \left (\zeta \div \BJ_h \right )
-i\omega \BE_h-i\omega\BKe,\bphi
\right )_{\CT_{{\rm m},h}}
-
\langle \zeta \div \BJ_h,\bphi \cdot \bn \rangle_{\partial \CT_{{\rm m},h}}
\\
&=
-\left (
-\omega^2 \al \BJ_h - \grad \left (\zeta \div \BJ_h \right )
-i\omega \BE_h-i\omega\BKe,\bphi
\right )_{\CT_{{\rm m},h}}
-
(\jmp{\zeta \div \BJ_h},\bphi \cdot \bn )_{\CF_{{\rm m},h}^{\rm i}}
\end{align*}
and
\begin{align*}
|b((\BE-\BE_h,\BJ-\BJ_h),(\bo,\bphi))|
\lesssim
\sum_{K\in\CT_{{\rm m},h}}\eta_{\ggrad,\ddiv,K}\sqrt{\zeta_K^\star}
\left(h_K^{-1}\|\bphi\|_K+h_K^{-1/2}\|\bphi\|_{\partial K}\right).
\end{align*}
Now, let $\tbw \in\BH^1(\CT_{{\rm m},h}) \cap \BH_0(\ddiv,\OOm)$.
By Galerkin orthogonality and the estimate \eqref{eq_qi_Hdm}, we conclude
\begin{align*}
|b((\BE\ -\ &\BE_h,\BJ-\BJ_h),(\bo,\tbw))|
=
|b((\BE-\BE_h,\BJ-\BJ_h),(\bo,\tbw-\CR_{{\rm m},h}\tbw)|
\\
&\lesssim
\sum_{K\in\CT_{{\rm m},h}}
\eta_{\ggrad,\ddiv,K}\sqrt{\zeta_K^\star}\left(
h_K^{-1}\|\tbw-\CR_{{\rm m},h}\tbw\|_K
+
h_K^{-1/2}\|(\tbw-\CR_{{\rm m},h}\tbw)\cdot\bn\|_{\partial K}
\right)
\\
&\lesssim
\sum_{K\in\CT_{{\rm m},h}}
\eta_{\ggrad,\ddiv,K}\sqrt{\zeta_K^\star}\|\grad\tbw\|_{\CT_{{\rm m},h}^K}
\lesssim
\eta_{\ggrad,\ddiv}\|\grad\tbw\|_{\zeta,\CT_{{\rm m},h}}.
\end{align*}
\end{proof}

We now establish that the proposed estimator is reliable in
Theorem \ref{theorem_rel}. The proof builds upon Lemmas \ref{lemma_rel_grad_o},
\ref{lemma_rel_o_curl}, \ref{lemma_rel_H1_o} and \ref{lemma_rel_o_H1} combined
with inf-sup condition \eqref{eq_inf_sup} and regular decomposition
\eqref{eq_regular_decomposition}.

\begin{theorem}
\label{theorem_rel}
We have
\begin{equation}
\label{eq_rel}
\enorm{(\BE-\BE_h,\BJ-\BJ_h)} \lesssim \eta.
\end{equation}
\end{theorem}

\begin{proof}
Since $\mathbb V$ in a Hilbert space, it follows from inf-sup
condition \eqref{eq_inf_sup} that there exists $(\bv^\star,\bw^\star) \in \mathbb V$
with $\enorm{(\bv^\star,\bw^\star)} = 1$ such that
\begin{equation*}
\enorm{(\BE-\BE_h,\BJ-\BJ_h)} \leq \cis^{-1} \Re b((\BE-\BE_h,\BJ-\BJ_j),\bv^\star,\bw^\star)).
\end{equation*}
Then, using \eqref{eq_regular_decomposition}, we have
\begin{equation*}
\bv^\star = \tbv + \grad p, \qquad \bw^\star = \tbw + \curl \bthe,
\end{equation*}
where $\tbv \in \BH^1(\CT_h) \cap \BH_0(\ccurl,\OO)$, $p \in H^1_0(\OO)$,
$\tbw \in \BH^1(\CT_{{\rm m},h}) \cap \BH_0(\ddiv,\OOm)$ and
$\bthe \in \BH^1(\CT_{{\rm m},h}) \cap \BH_0(\ccurl,\OOm)$
with
\begin{equation*}
\omega \|\grad q\|_{\varepsilon,\OO} + \omega \|\grad_h \bthe\|_{\alpha,\OOm}
+
\|\grad_h \tbv\|_{\chi,\OO} + \|\grad_h \tbw\|_{\zeta,\OOm}
\lesssim \enorm{(\bv^\star,\bw^\star)} = 1,
\end{equation*}
and \eqref{eq_rel} follows by linearity and the estimates established in
Lemma \ref{lemma_rel_grad_o}, \ref{lemma_rel_o_curl}, \ref{lemma_rel_H1_o}
and \ref{lemma_rel_o_H1}.
\end{proof}

\subsection{Efficiency}

We show that the proposed estimator is efficient. To this end, we establish four results that provide upper bounds for each of the four terms constituting our estimator.

\begin{lemma}
The estimate
\begin{equation}
\label{eq_eff_div}
\eta_{\ddiv,K}
\lesssim
(1 + \kPK h_K)
\!\enorm{(\BE-\BE_h,\BJ-\BJ_h)}_{\tK} + \osc_{\tK},
\end{equation}
holds true for all $K \in \CT_h$.
\end{lemma}

\begin{proof}
Using \eqref{eq_wavenumbers}, we first record that the estimate
\begin{equation}
\label{tmp_div_BJ_BJh}
\frac{h_K}{\sqrt{\eps_K^\star}} \|\div(\BJ-\BJ_h)\|_K
\lesssim
\kPK h_K \enorm{(\BE-\BE_h,\BJ-\BJ_h)}_K,
\end{equation}
holds true for any $K \in \CT_h$.

We then fix $K \in \CT_h$, and introduce the notations
$\br_K \eq i\omega \ee \BE_h + \BJ_h - \BJe$
and $\br_K^h \eq \pi_h \br_K$. Recalling \eqref{eq_hidden_div},
we have
\begin{equation*}
-\div \br_K^h
=
i\omega \div (\ee(\BE-\BE_h)) + \div (\BJ-\BJ_h) - \div (\BJe-\pi_h\BJ_e),
\end{equation*}
and thanks to \eqref{eq_norm_bubble}
\begin{align}
\label{tmp_div_brKh}
&\|\div \br_K^h\|_{0,K}^2
\lesssim
|(b_K \div \br_K^h,\div \br_K^h)_K|
\\ \nonumber
&\lesssim
|(b_K \div \br_K^h,\div (i\omega\ee(\BE-\BE_h)))_K|
+
|(b_K \div \br_K^h,\div (\BJ-\BJ_h)+ \div (\BJe-\pi_h\BJe))_K|.
\end{align}
Then, we use \eqref{eq_inv_bubble} to estimate the two terms in the right-hand side of
\eqref{tmp_div_brKh} with
\begin{equation*}
|(b_K \div \br_K^h,\div (i\omega \ee(\BE-\BE_h)))_K|
=
|(\grad (b_K \div \br_K^h),\omega\ee(\BE-\BE_h))_K|
\lesssim
h^{-1}_K\|\div \br_K^h\| \omega\|\ee(\BE-\BE_h)\|_K
\end{equation*}
and
\begin{equation*}
|(b_K \div \br_K^h,\div (\BJ-\BJ_h)+ \div (\BJe-\pi_h\BJe))_K|
\lesssim
\|\div \br_K^h\|_K
\left (
\|\div (\BJ-\BJ_h)\|_K + \|\div (\BJe-\pi_h\BJe)\|_K
\right ),
\end{equation*}
and it follows that
\begin{equation*}
\frac{h_K}{\sqrt{\eps_K^\star}} \|\div \br_K^h\|_K
\lesssim
\omega \|\BE-\BE_h\|_{\ee,K}
+
\frac{h_K}{\sqrt{\eps_K^\star}} \|\div(\BJ-\BJ_h)\|_K
+
\frac{h_K}{\sqrt{\eps_K^\star}} \|\div (\BJe-\pi_h \BJe)\|_K.
\end{equation*}
After observing $\br_K^h - \br_K = \BJe-\pi_h \BJe$, we conclude that
\begin{equation}
\label{tmp_eff_div_K}
\frac{h_K}{\sqrt{\eps_K^\star}} \|\div \br_K\|_K
\lesssim
(1+\kPK h_K) \enorm{(\BE-\BE_h,\BJ-\BJ_h)}_K + \osc_K.
\end{equation}
Now, if $F \in \CF_K$, we let $w_F \eq b_F \LL_F(\jmp{\ee\BE_h} \cdot \bn_F)$.
Since $\ee\BE \in \BH(\ddiv,\tF)$ and $ w_F\in H_0^1(\tF)$, we
can employ integration by parts and \eqref{eq_norm_bubble} to show that
\begin{equation*}
\|\jmp{\ee\BE_h} \cdot \bn_F\|_F^2
\lesssim
|\langle\jmp{\ee\BE_h} \cdot \bn_F,w_F\rangle_F|
=
|(\ee(\BE-\BE_h),\grad w_F)_{\tF} + (\divh (\ee (\BE-\BE_h)),w_F)_{\tF}|,
\end{equation*}
and it follows from \eqref{eq_ext_bubble} that
\begin{equation}
\label{tmp_eff_div_F}
\frac{\omega h_K^{1/2}}{\sqrt{\eps_K^\star}}
\|\jmp{\ee\BE_h}\cdot \bn_F\|_F
\lesssim
\omega\|\BE-\BE_h\|_{\ee,\tF}
+
\frac{\omega h_K}{\sqrt{\eps_K^\star}} \|\divh (\ee (\BE-\BE_h))\|_{\tF}.
\end{equation}
Recalling \eqref{eq_hidden_div}, we have
\begin{equation*}
\divh (i\omega\ee(\BE-\BE_h))
=
\divh (i\omega\ee\BE_h + \BJ_h-\BJe) + \div (\BJ-\BJ_h),
\end{equation*}
and \eqref{eq_wavenumbers}, \eqref{tmp_div_BJ_BJh} as well as \eqref{tmp_eff_div_K}
show that
\begin{equation}
\label{tmp_eff_div_F2}
\frac{\omega h_K}{\sqrt{\eps_K^\star}} \|\divh(\ee(\BE-\BE_h))\|_{\tF}
\lesssim
(1+ \kPK h_K) \enorm{(\BE-\BE_h,\BJ-\BJ_h)}_{\tF}+\osc_{\tK}.
\end{equation}
Then, \eqref{eq_eff_div} follows from \eqref{tmp_eff_div_K}, \eqref{tmp_eff_div_F}
and \eqref{tmp_eff_div_F2}.
\end{proof}

\begin{lemma}
The estimate
\begin{equation}
\label{eq_eff_curl}
\eta_{\ccurl,K}
\lesssim
\left(1 + \frac{\cJK}{\cEK} \kPK h_K \right)
\!\enorm{(\BE-\BE_h,\BJ-\BJ_h)}_{\tK}
\end{equation}
holds true for all $K \in \CT_h$.
\end{lemma}

\begin{proof}
Thanks to \eqref{eq_hidden_curl}, \eqref{eq_norm_bubble} and after integrating by parts,
we obtain that
\begin{align}
\label{tmp_curl_BJh_BEh}
&\|\curl(i\omega\al\BJ_h-\BE_h-\varrho_h\BKe)\|_K^2
\\ \nonumber
&\lesssim
|(b_K\curl(i\omega\al\BJ_h-\BE_h-\varrho_h\BKe),\curl(i\omega\al\BJ_h-\BE_h-\varrho_h\BKe)_K|
\\ \nonumber
&\leq
|(\curl(b_K\curl(i\omega\al\BJ_h-\BE_h-\varrho_h\BKe)),i\omega\al(\BJ-\BJ_h))_K|
\\
&+|(b_K\curl(i\omega\al\BJ_h-\BE_h-\varrho_h\BKe),\curl(\BE-\BE_h)+\curl(\BKe-\varrho_h\BKe))_K|\nonumber
\end{align}
for all $K \in \CT_h$.

Now, for the terms in the right-hand side of \eqref{tmp_curl_BJh_BEh}, we use \eqref{eq_inv_bubble} and have
\begin{equation*}
|(\curl(b_K\curl(i\omega\al\BJ_h-\BE_h-\varrho_h\BKe)),i\omega\al(\BJ-\BJ_h))_K|
\lesssim
h_K^{-1}\|\curl(i\omega\al\BJ_h-\BE_h-\varrho_h\BKe)\|_K\omega\|\al(\BJ-\BJ_h)\|_K
\end{equation*}
and
\begin{align*}
|(b_K\curl(i\omega\al\BJ_h-\BE_h-\varrho_h\BKe)&,\curl(\BE-\BE_h)+\curl(\BKe-\varrho_h\BKe))_K|
\\
&\lesssim
\|\curl(i\omega\al\BJ_h-\BE_h-\varrho_h\BKe)\|_K\|\curl(\BE-\BE_h)\|_K
\\
&+\|\curl(i\omega\al\BJ_h-\BE_h-\varrho_h\BKe)\|_K\|\curl(\BKe-\varrho_h\BKe)\|_K.
\end{align*}
Then,
\begin{multline*}
\frac{h_K}{\sqrt{\alpha_K^\star}}\|\curl(i\omega\al\BJ_h-\BE_h-\varrho_h\BKe)\|_K
\lesssim
\\
\omega\|\BJ-\BJ_h\|_{\al,K}
+
\frac{h_K}{\sqrt{\chi_K^\star\alpha_K^\star}}\|\curl(\BE-\BE_h)\|_{\cc,K}
+
\frac{h_K}{\sqrt{\chi_K^\star\alpha_K^\star}}\|\curl(\BKe-\varrho_h\BKe)\|_{\cc,K},
\end{multline*}
and recalling \eqref{eq_wavenumbers}, we conclude that
\begin{align}\label{tmp_eff_curl_K}
\frac{h_K}{\sqrt{\alpha_K^\star}}\|\curl(i\omega\al\BJ_h-\BE_h-\BKe)\|_K
\lesssim
\left(1+\frac{\cJK}{\cEK}\kPK h_K\right)\!\enorm{(\BE-\BE_h,\BJ-\BJ_h)}_K+\osc_K.
\end{align}
On the other hand, for a fixed $K \in \CT_h$, if $F \in \CF_K$ we set
$\bw_F \eq b_F \LL_F(\jmp{\al\BJ_h} \times \bn_F)$. Recalling that $\al \BJ \in\BH(\ccurl,\tF)$
and $b_F \in H^1_0(\tF)$, and using \eqref{eq_norm_bubble}, we have
\begin{align*}
\|\jmp{\al\BJ_h}\times\bn_F\|_F^2&\lesssim|\langle\jmp{\al\BJ_h}\times\bn_F,\bw_F\rangle_F|\\
&=
|(\al(\BJ-\BJ_h),\curl\bw_F)_{\tF}+(\curl(\al(\BJ-\BJ_h)),\bw_F)_{\tF}|,
\end{align*}
and thanks to \eqref{eq_ext_bubble}, we get that
\begin{align}\label{tmp_eff_curl_F}
\frac{\omega h_K^{1/2}}{\sqrt{\alpha_K^\star}}
\|\jmp{\al\BJ_h}\times\bn_F\|_F
\lesssim
\omega\|\BJ-\BJ_h\|_{\al,\tF}
+
\frac{\omega h_K}{\sqrt{\alpha_K^\star}}
\|\curl(\al(\BJ-\BJ_h))\|_{\tF}.
\end{align}
Invoking \eqref{eq_hidden_curl}, we have
\begin{align*}
\curl(i\omega\al(\BJ-\BJ_h))
=
-\curl(i\omega\al\BJ_h-\BE_h-\BKe)
+\curl(\BE-\BE_h)
\end{align*}
and then, \eqref{eq_wavenumbers} and \eqref{tmp_eff_curl_K} let us conclude that
\begin{align}\label{tmp_eff_curl_F2}
\frac{\omega h_K}{\sqrt{\alpha_K^\star}}\|\curl(\al(\BJ-\BJ_h))\|_{\tF}
\lesssim
\left(1+\frac{\cJK}{\cEK}\kPK h_K\right)\!\enorm{(\BE-\BE_h,\BJ-\BJ_h)}_{\tF}
+
\osc_{\tK}.
\end{align}
Finally, \eqref{eq_eff_curl} follows from \eqref{tmp_eff_curl_K}, \eqref{tmp_eff_curl_F}
and \eqref{tmp_eff_curl_F2}.
\end{proof}

\begin{lemma}
The estimate
\begin{align}
\label{eq_eff_curl_curl}
\eta_{\ccurl,\ccurl,K}
\lesssim
\left(1+\left(\kEK+\frac{\cJK}{\cEK}\kPK\right)h_K\right)
\!\enorm{(\BE-\BE_h,\BJ-\BJ_h)}_{\tK}+\osc_{\tK}
\end{align}
holds true for all $K \in \CT_h$.
\end{lemma}

\begin{proof}
For a fixed $ K\in\CT_h $, we set $ \br_K:=-\omega^2 \ee \BE_h+\curl(\cc\curl \BE_h)+i\omega\BJ_h-i\omega\BJ_{\rm e} $ and $\br_K^h:=\pi_h\br_K$. Then, considering \eqref{eq_maxwell_drude_strong}, \eqref{eq_norm_bubble} and integrating by parts, we get 
\begin{multline}\label{tmp_curl_curl_brKh}
\|\br_K^h\|_K^2
\lesssim
|(\br_K^h,b_K\br_K^h)_K| \\
\lesssim
|(\omega^2\ee(\BE-\BE_h),b_K\br_K^h)_K|+|(\cc\curl(\BE-\BE_h),\curl(b_K\br_K^h))_K| \\
+|(i\omega(\BJ-\BJ_h),b_K\br_K^h)_K|+|(i\omega(\BJe-\pi_h\BJe),b_K\br_K^h)_K|.
\end{multline}
To bound the right-hand side terms in \eqref{tmp_curl_curl_brKh}, we use \eqref{eq_inv_bubble} to obtain
\begin{align*}
|(\omega^2\ee(\BE-\BE_h),b_K\br_K^h)_K|
&\lesssim
\omega^2\|\ee(\BE-\BE_h)\|_K\|\br_K^h\|_K,
\\
|\cc\curl(\BE-\BE_h),\curl(b_K\br_K^h))_K|
&
\lesssim
h_K^{-1}\|\cc\curl(\BE-\BE_h)\|_K\|\br_K^h\|_K,
\\
|(i\omega(\BJ-\BJ_h),b_K\br_K^h)_K|
&\lesssim
\omega\|\BJ-\BJ_h\|_K\|\br_K^h\|_K,
\\
|(i\omega(\BJe-\pi_h\BJe),b_K\br_K^h)_K|
&\lesssim
\omega\|\BJe-\pi_h\BJe\|_K\|\br_K^h\|_K,
\end{align*}
and we have that
\begin{align*}
\frac{h_K}{\sqrt{\chi_K^\star}}\|\br_K^h\|_K
&\lesssim
\frac{\omega^2 h_K}{\sqrt{\chi_K^\star}}\|\ee(\BE-\BE_h)\|_K
+
\frac{1}{\sqrt{\chi_K^\star}}\|\cc\curl(\BE-\BE_h)\|_K
+
\frac{\omega h_K}{\sqrt{\chi_K^\star}}\|\BJ-\BJ_h\|_K
+
\osc_K.
\end{align*}
Noticing that $\br_K^h-\br_K = i\omega(\BJe-\pi_h\BJe)$ and recalling \eqref{eq_wavenumbers},
we get
\begin{align}\label{tmp_eff_curl_curl_K}
\frac{h_K}{\sqrt{\chi_K^\star}}\|\br_K\|_K
&\lesssim
\left(1+\left(\kEK+\frac{\cJK}{\cEK}\kPK\right) h_K\right)
\enorm{(\BE-\BE_h,\BJ-\BJ_h)}_{K}
+
\osc_K.
\end{align}
Now, for $F \in \CF_h$, we define $\bw_F \eq b_F \LL_F(\jmp{\ee\BE_h} \cdot \bn_F)$.
Since $\cc\curl\BE\in\BH(\ccurl,\tF)$ and $\bw_F \in \BH_0^1(\tF)$, using \eqref{eq_norm_bubble}
and integrating by parts, we see that
\begin{align*}
\|\jmp{\cc\curl\BE_h}\times\bn_F\|_F^2
&\lesssim
|\langle\jmp{\cc\curl\BE_h}\times\bn_F,\bw_F\rangle_F|
\\
&=
|(\curl(\cc\curl(\BE-\BE_h)),\bw_F)_{\tF}-(\cc\curl(\BE-\BE_h),\curl\bw_F)_{\tF}|.
\end{align*}
Thanks to \eqref{eq_maxwell_drude_strong}, we deduce that
\begin{align*}
\|\jmp{\cc\curl\BE_h}\times\bn_F\|_F^2
&\lesssim
|(\omega^2\ee(\BE-\BE_h),\bw_F)_{\tF}-(\cc\curl(\BE-\BE_h),\curl\bw_F)_{\tF}
\\
&\quad\ -(i\omega(\BJ-\BJ_h),\bw_F)_{\tF}+(\br_K,\bw_K)_{\tF}|
\end{align*}
and then, thanks to \eqref{eq_ext_bubble}, we have that
\begin{align*}
\frac{h_K^{1/2}}{\sqrt{\chi_K^\star}}\|\jmp{\cc\curl\BE_h}\times\bn_F\|_F
&\lesssim
\frac{\omega^2 h_K}{\sqrt{\chi_K^\star}}\|\ee(\BE-\BE_h)\|_{\tF}
+
\frac{1}{\sqrt{\chi_K^\star}}\|\cc\curl(\BE-\BE_h)\|_{\tF}
\\
&\;+
\frac{\omega h_K}{\sqrt{\chi_K^\star}}\|\BJ-\BJ_h\|_{\tF}
+
\frac{h_K}{\sqrt{\chi_K^\star}}\|\br_K\|_{\tF}.
\end{align*}
Finally, \eqref{eq_wavenumbers} shows that
\begin{align}\label{tmp_eff_curl_curl_F}
\frac{h_K^{1/2}}{\sqrt{\chi_K^\star}}\|\jmp{\cc\curl\BE_h}\times\bn_F\|_F
&\lesssim
\left(1+\left(\kEK+\frac{\cJK}{\cEK}\kPK\right)\!h_K\right)\!\enorm{(\BE-\BE_h,\BJ-\BJ_h)}_{\tF}
\\ \nonumber
&+\;\frac{h_K}{\sqrt{\chi_K^\star}}\|\br_K\|_{\tF}
\end{align}
and hence \eqref{eq_eff_curl_curl} is a direct consequence of \eqref{tmp_eff_curl_curl_K} and \eqref{tmp_eff_curl_curl_F}.
\end{proof}

\begin{lemma}
The estimate
\begin{align}\label{eq_eff_grad_div}
\eta_{\ggrad,\ddiv,K}
\lesssim(1+(\kEK+\kJK)h_K)\!\enorm{(\BE-\BE_h,\BJ-\BJ_h)}_{\tK}+\osc_{\tK}
\end{align}
holds true for all $K \in \CT_h$.
\end{lemma}
\begin{proof}
For a fixed $K \in \CT_h$, let
$\br_K \eq -\omega^2 \al \BJ_h -\grad \left (\zeta \div \BJ_h \right )
-i\omega\BE_h-i\omega\varrho_h\BKe$ and $\br_K^h \eq \varrho_h \br_K$.
Then, after integrating by parts and thanks to \eqref{eq_norm_bubble}, we have that
\begin{align*}
\|\br_K^h\|_K^2
&\lesssim
|(\br_K^h,b_K\br_K^h)_K|
\\
&=
|(\omega^2\al(\BJ-\BJ_h),b_K\br_K^h)_K
+
(\zeta\div(\BJ-\BJ_h),\div(b_K\br_K^h))_K
+
(i\omega(\BE-\BE_h),b_K\br_K^h)_K
\\
&+\;
(i\omega(\BKe-\varrho_h\BKe),b_K\br_K^h)_K|
\end{align*}
and thus, using \eqref{eq_inv_bubble}, we obtain
\begin{align*}
\frac{h_K}{\sqrt{\zeta_K^\star}}\|\br_K^h\|_K
\lesssim
\frac{\omega^2 h_K}{\sqrt{\zeta_K^\star}}\|\al(\BJ-\BJ_h)\|_K
+
\frac{h_K}{\sqrt{\zeta_K^\star}}\|\zeta\div(\BJ-\BJ_h)\|_K
+
\frac{\omega h_K}{\sqrt{\zeta_K^\star}}\|\BE-\BE_h\|_K
+
\osc_K.
\end{align*}
Recalling \eqref{eq_wavenumbers}, we get that
\begin{align}\label{tmp_eff_grad_div_K}
\frac{h_K}{\sqrt{\zeta_K^\star}}\|\br_K\|_K
\lesssim
(1+(\kJK+\kPK)h_K) \enorm{(\BE-\BE_h,\BJ-\BJ_h)}_{K}
+
\osc_K.
\end{align}
On the other hand, for $F \in \CF_K$, we set $\bw_F \eq b_F\LL_F(\jmp{\zeta\div\BJ_h}\bn_F)$.
Thanks to the fact that $\zeta \div\BJ \in H^1(\tF)$
and $b_F \in H^1_0(\tF)$, estimate \eqref{eq_norm_bubble} and integration
by parts reveal that
\begin{equation*}
\|\jmp{\zeta\div\BJ_h}\|_F^2
\lesssim
|\langle\jmp{\zeta\div\BJ_h},\bw_F\rangle_F|
=
|(\grad(\zeta\div(\BJ-\BJ_h)),\bw_F)_{\tF}
+
(\zeta\div(\BJ-\BJ_h),\div\bw_F)_{\tF}|.
\end{equation*}
Using \eqref{eq_maxwell_drude_strong}, we have that
\begin{align*}
&\|\jmp{\zeta\div\BJ_h}\|_F^2
\lesssim
\\
&
|-(\omega^2\al(\BJ-\BJ_h),\bw_F)_{\tF}
+(\zeta\div(\BJ-\BJ_h),\div\bw_F)_{\tF}
-
(i\omega(\BE-\BE_h),\bw_F)_{\tF}
+
(\br_K^h,\bw_F)_{\tF}|
\end{align*}
and using \eqref{eq_ext_bubble}
\begin{align*}
\frac{h_K^{1/2}}{\sqrt{\zeta_K^\star}}\|\jmp{\zeta\div\BJ_h}\|_F
&\lesssim
\frac{\omega^2 h_K}{\sqrt{\zeta_K^\star}}\|\al(\BJ-\BJ_h)\|_{\tF}
+
\frac{1}{\sqrt{\zeta_K^\star}}\|\zeta\div(\BJ-\BJ_h)\|_{\tF}
\\
&+\;
\frac{\omega h_K}{\sqrt{\zeta_K^\star}}\|\BE-\BE_h\|_{\tF}
+
\frac{h_K}{\sqrt{\zeta_K^\star}}\|\br_K\|_{\tF}.
\end{align*}
Then, \eqref{eq_wavenumbers} let us conclude that
\begin{align}\label{tmp_eff_grad_div_F}
\frac{h_K^{1/2}}{\sqrt{\zeta_K^\star}}\|\jmp{\zeta\div\BJ_h}\|_F
&\lesssim
(1+(\kJK+\kPK)h_K)\enorm{(\BE-\BE_h,\BJ-\BJ_h)}_{\tF}
+
\frac{h_K}{\sqrt{\zeta_K^\star}}\|\br_K\|_{\tF},
\end{align}
and \eqref{eq_eff_grad_div} follows from \eqref{tmp_eff_grad_div_K} and \eqref{tmp_eff_grad_div_F}.
\end{proof}

\begin{theorem}
\label{theorem_efficiency}
The estimate
\begin{equation}
\label{eq_efficiency}
\eta_K
\lesssim
\left(
1 + \left(\kEK+\kJK+\left(1+\frac{\cJK}{\cEK}\right)\kPK \right)h_K
\right)
\enorm{(\BE-\BE_h,\BJ-\BJ_h)}_{\tK}
+
\osc_{\CT_{K,h}}
\end{equation}
holds true for all $ K\in\CT_h $.
\end{theorem}

\section{Numerical examples}
\label{section_numerics}


\subsection{Settings}

We first present the settings and methology common to our three examples.

\subsubsection{Two-dimensional notations}

Our numerical experiments are performed in a two-dimensional setting.
We thus assume that the last component of the fields $\BE$, $\BJ$,
$\BJe$ and $\BKe$ vanishes and that the first two components only depend
on the $(\bx_1,\bx_2)$ space variables. We further assume that the
coefficients take the form
\begin{equation*}
\ee = \left (
\begin{array}{ccc}
\ee_{11} & \ee_{12} & 0
\\
\ee_{21} & \ee_{22} & 0
\\
0        & 0        & \ee_{33}
\end{array}
\right )
\quad
\cc = \left (
\begin{array}{ccc}
\cc_{11} & \cc_{12} & 0
\\
\cc_{21} & \cc_{22} & 0
\\
0        & 0        & \chi
\end{array}
\right ) \quad
\al = \left (
\begin{array}{ccc}
\al_{11} & \al_{12} & 0
\\
\al_{21} & \al_{22} & 0
\\
0        & 0        & \al_{33}
\end{array}
\right ),
\end{equation*}
leading to
\begin{equation}
\label{eq_maxwell_drude_2D}
\left \{
\begin{array}{rcll}
-\omega^2 \ee \BE + \vcurl \left (\chi \scurl \BE \right ) + i\omega \BJ
&=&
\BJe
&
\text{ in } \OO,
\\
-\omega^2 \al \BJ - \grad \left (\zeta \div \BJ\right ) - i\omega \BE
&=&
\BKe
&
\text{ in } \OOm,
\end{array}
\right .
\end{equation}
where the boldface notation now stands for two-components vectors and tensors.
As usual, the two-dimensional curl operators are given by
\begin{equation*}
\scurl \bphi = \partial_1 \bphi_2 - \partial_2 \bphi_1
\qquad
\vcurl \; \phi = (\partial_2 \phi,-\partial_1 \phi).
\end{equation*}

\subsubsection{Perfectly matched layers}

We employ perfectly matched layers to incorporate the radiation condition
into a bounded computational domain. In our examples, we assume for the
sake of simplicity that $\ee$ is diagonal and that the metallic particles
are contained into a box $\Omega_0 \eq (-L,L)^2$ for some $L > 0$. We enclose
$\Omega_0$ into a larger box $\Omega \eq (-L-\ell,L+\ell)^2$ featuring an additional
layer of size $\ell > 0$. Following \cite{monk_2003a}, we define
\begin{equation*}
d_j(\bx) \eq 1 - \frac{3}{4} i \boldsymbol 1_{|\bx_j| > L}
\end{equation*}
for $j=1$ or $2$, and construct modified coefficients
\begin{equation*}
\widetilde \ee
\eq
\left (
\begin{array}{cc}
d_2/d_1 & 0
\\
0       & d_1/d_2
\end{array}
\right )
\ee
\qquad
\widetilde \chi = \frac{\chi}{d_1 d_2}.
\end{equation*}
These new coefficients are actually unchanged in $\Omega_0$, but take
artificial values in the additional layer designed to absorb incoming
radiations without spurious reflections. In the remaining of this section,
we employ the artificial coefficients $\widetilde \ee$ and $\widetilde \chi$,
but omit the $\widetilde \cdot$ notation to ease the presentation.

\subsubsection{Incident field injection}

We consider the scattering of an incident plane wave
by metallic nanostructures. The total field $\BE^{\rm t}$ splits into
the (known) incident field $\BE^{\rm i}$ and the scattered field $\BE$ that we
numerically approximate. We decompose the computational domain as
$\OO = \OOm \cup \OOo \cup \OOp$, where
$\OOm$ corresponds to the metallic inclusions, $\OOp$ is the PML region, and
$\OOo \eq \OO \setminus {\overline{\OOm} \cup \overline{\OOp}}$. $\BE^{\rm i}$
is a solution to Maxwell's equations in $\OOo$. $\BE$ is a scattered field that
satisfies the PML equation inside $\OOp$. Finally, the total field $\BE^{\rm t}$ satisfies the Maxwell-Drude
system in $\OOo \cup \OOm$. It follows that the pair $(\BE,\BJ)$ is solution to
\eqref{eq_maxwell_drude_2D} with $\BJe \eq \bo$ and $\BKe \eq i\omega\BE^{\rm i}$.
In the forthcoming examples, we will only consider right-hand sides of this form,
where
\begin{equation*}
\BE^{\rm i}(\bx) = \bp e^{-ik \bd \cdot \bx},
\end{equation*}
where $\bp,\bd$ are two unit vectors such that $\bp \cdot \bx$ and
$k \eq \omega/c_0$.  $\bp$ and $\bd$ respectively
describe the polarization and the direction of the incident wave, while
$c_0 \eq \sqrt{\varepsilon_0\mu_0}$ is the speed of light and $k$, the
wavenumber.

\subsubsection{Coefficients}

The permittivity and permeability are set to the vacuum values
in $\OOo$, that is
\begin{equation*}
\ee = \BI \eps_0, \qquad \chi = \frac{1}{\mu_0},
\end{equation*}
with the aforementioned modification in the PML region $\OOp$. In the metal,
the coefficients $\al$ and $\zeta$ are defined from $\omegaP$, $\gamma$ and
$\vartheta_{\rm F}$ by
\begin{equation*}
\al \eq \frac{1}{\omegaP^2 \eps_0} \left(1 - \frac{\gamma}{i\omega}\right ) \BI,
\qquad
\zeta \eq \frac{3}{5} \frac{\vartheta_{\rm F}^2}{\omegaP^2 \eps_0}.
\end{equation*}
The actual values of $\omegaP$, $\gamma$ and $\vartheta_{\rm F}$ depend on the particular metal
under consideration. For gold, we have
\begin{equation*}
\omegaP \eq 1.390 \cdot 10^{16} \;\text{rad} \cdot \text{s}^{-1},
\qquad
\gamma  \eq 3.230 \cdot 10^{13} \; \text{rad} \cdot \text{s}^{-1},
\qquad
\varphi_{\rm F} \eq 1.084 \cdot 10^6 \; \text{m} \cdot \text{s}^{-1},
\end{equation*}
while for silver, we employ
\begin{equation*}
\omegaP \eq 1.339 \cdot 10^{16} \;\text{rad} \cdot \text{s}^{-1},
\qquad
\gamma  \eq 1.143 \cdot 10^{14} \; \text{rad} \cdot \text{s}^{-1},
\qquad
\varphi_{\rm F} \eq 1.465 \cdot 10^6 \; \text{m} \cdot \text{s}^{-1}.
\end{equation*}

\subsubsection{Adaptive algorithm}

In the following examples, we employ the estimator described before to
steer an adaptive mesh algorithm process. We fix once and for all the
polynomial degree $p$ and start with an initial mesh $\CT_h^{(0)}$. Then,
assuming we arrived at a mesh $\CT_h^{(\ell)}$, we solve the finite element
system associated with this mesh, and compute the associated elementwise
error estimators $\eta_K$. These estimators are in turn use to output
a new mesh $\CT_h^{(\ell+1)}$, enabling the start of new iteration.
We employ the software packages {\tt MUMPS} \cite{amestoy_duff_lexcellent_2000a}
to solve the linear systems, and {\tt MMG} \cite{mmg3d} to generate the meshes.
Algorithm \ref{algo_adaptive} describes the resulting procedure. Notice that {\tt MMG}
refines an existing mesh $\CT$ by following new local mesh sizes $\bh_{\ba}$ that are
given on the vertices $\ba$ of $\CT$. As a result, Algorithm \ref{algo_adaptive} includes
a ``translation'' between the ``element-based'' estimator $\eta_K$ and the data $\bh_{\ba}$
passed to {\tt MMG}.

\begin{algorithm}
\begin{algorithmic}[1]
\Procedure{generate\_mesh}{$\CT$, $\bh$}
\State generate $\widetilde \CT$ by calling {\tt MMG} with the input mesh $\CT$ and the vertex mesh size $\bh$
\State \Return $\widetilde \CT$
\EndProcedure
\Procedure{generate\_mesh\_sizes}{$\CT$, $\eta$, $\theta$, $\rho$}
\State Let $n_{\rm v}$ denote the number of vertices of $\CT$
\State zero initialize arrays $\bet$ and $\bh$ of size $n_{\rm v}$
\ForEach{mesh element $K \in \CT$}
\ForEach{element vertex $\ba \in \CV_K$}
\State $\bet[\ba] = \bet[\ba] + \eta_K$
\State $\bh[\ba] = \max(\bh[\ba],h_K)$
\EndFor
\EndFor
\State sort the vertices in an array ord such that $\bet[\text{ord}[j]]$ is non-decreasing
\State find the smallest integer $n \in \{1,\dots,n_{\rm v}\}$ such that
$\sum_{j=1}^{n} \bet[\text{ord}[j]]^2 \geq \theta \sum_{\ba} \bet[\ba]^2$.
\For{$j=1,\dots,n$}
\State $\bh[\text{ord}[i]] = \rho \bh[\text{ord}[i]]$
\EndFor
\State \Return $\bh$
\EndProcedure
\Procedure{adaptive\_loop}{$\CT^{(0)}$,$\ell_{\rm max}$,$\theta$,$\rho$}
\For{$\ell=0,\dots,\ell_{\rm max}$}
\State assemble the finite-element matrix associated with $\CT^{(0)}$
\State solve the linear system with {\tt MUMPS}
\State compute the estimator $\eta$
\State compute the new mesh sizes $\bh = \Call{generate\_mesh\_sizes}{\CT^{(\ell)},\eta,\theta,\rho}$
\State generate the new mesh $\CT^{(\ell+1)} = \Call{generate\_mesh}{\CT^{(\ell)},\bh}$
\EndFor
\EndProcedure
\end{algorithmic}
\caption{Adaptive loop}
\label{algo_adaptive}
\end{algorithm}

The adaptive procedure takes two additional parameters $\theta$ and $\rho$ that
controls how many elements are refined at each iteration, and how much their sizes
is reduced. In the examples below, we always select $\theta \eq 0.05$ and $\rho \eq 0.5$.
While we mean that we refine elements that contribute to $5\%$ of the total squared
error, and that these elements have their diameter divided by two.

\subsubsection{Error measurements}

The analytical solutions for the examples below are not available,
which complicates the numerical validation of the proposed error estimator.
For a given mesh and polynomial $p$, if $(\BE_h,\BJ_h)$ denotes
the computed discrete solution, we compute a ``reference'' solution
$(\widetilde \BE_h,\widetilde \BJ_h)$ on the same mesh with $\widetilde p = p+2$.
We then employ the quantities
\begin{equation*}
\xi_K \eq \enorm{(\widetilde \BE_h-\BE_h,\widetilde \BJ_h-\BJ_h)}_K
\qquad
\xi^2 \eq \sum_{K \in \CT_h} \xi_K^2
\end{equation*}
to obtain a measure of the discretization error.

\subsubsection{Comparison with uniform meshes}

We also benchmark the adaptive process against uniform meshes. To this
end, we  build for  each geometry  of interest  a sequence  of uniform
meshes  with {\tt  MMG} by  simply  requiring a  maximal allowed  mesh
size. The mesh size is chosen  so that the resulting number of degrees
of  freedom  is similar  to  the  structured  meshes produced  by  the
adaptive algorithm. This enables  to quantify the accuracy improvement
due to local refinements, since roughly the same computational cost is
then required for the structured and the unstructured meshes. To avoid
any  confusion,  we  employ  below  the  notation  $\xi_{\rm  u}$  and
$\eta_{\rm u}$ for  the error and estimator computed  with the uniform
meshes, while the  quantities $\xi_{\rm a}$ and  $\eta_{\rm a}$ relate
to the adaptive meshes.

\subsection{Gold bowtie antenna}

Our first example is a bowtie nano-antenna made of gold, as depicted on Figure
\ref{figure_bowtie_settings}. The incidence angle is $\theta = \pi/3$, and thus
$\bd = (\cos\theta,\sin\theta)$ and $\bp = (-\sin\theta,\cos\theta)$. We consider
three frequencies, namely $\omega = 0.8\omegaP$, $0.9\omegaP$ and $\omegaP$.
Figure \ref{figure_bowtie_sol} presents the reference solution computed on the
finest mesh. The case where $\omega=\omegaP$ is of particular interest: it can
been seen on Figure \ref{figure_bowtie_sol} that the desired light-focusing effect
is effectively achieved.

We start the adaptive loop with the initial mesh on the right panel of
Figure \ref{figure_bowtie_settings} and run this loop for $80$ iterations
with the polynomial degree $p=1$. Figure \ref{figure_bowtie_convergence} depicts 
the behaviour of the errors $\xi_{\rm a}$ and $\xi_{\rm u}$ plotted against
the number of degrees of freedom $N_{\rm dofs}$. The accuracy is significantly
improved on adaptive meshes for a similar number of degrees of freedom. Besides,
we observe the optimal convergence rate in $N_{\rm dofs}^{-(p+1)/2}$, which means 
that the estimator correctly steers the mesh refinement process.
Figure \ref{figure_bowtie_effectivity} shows the effectivity index of the
estimator for both adaptive and uniform meshes. The effectivity index first
oscillates before stabilizing asymptotically for fine meshes. This behaviour
is typical of non-coercive problems \cite{chaumontfrelet_vega_2020a}. It is
also in agreement with efficiency estimate \eqref{eq_efficiency} of Theorem
\ref{theorem_efficiency} which states that the estimator may become inefficient
on coarse meshes.
Finally, we present the elementwise actual error distribution and the
estimator $\eta_K$ in the central region of the mesh for $\omega = 0.8\omegaP$
in Figure \ref{figure_el_bowtie_s080}. While the scales of the left and
right panels are different, the (relative) agreement between the actual error
and the estimator is excellent.

\begin{figure}

\begin{minipage}{.45\linewidth}
\begin{tikzpicture}[scale=.3]

\draw[ultra thick,pattern=north west lines, pattern color=black]
(-3.00, 3.00) -- (-3.00,-3.00) -- (-0.25,-0.25) -- ( 0.25,-0.25) --
( 3.00,-3.00) -- ( 3.00, 3.00) -- ( 0.25, 0.25) -- (-0.25, 0.25) -- cycle;

\draw[thick] (-10,-10) rectangle (10,10);

\draw[dashed] (- 6,-10) -- (- 6, 10);
\draw[dashed] (  6,-10) -- (  6, 10);
\draw[dashed] (-10,- 6) -- ( 10,- 6);
\draw[dashed] (-10,  6) -- ( 10,  6);

\draw[<->] (-3,-3.5) -- (3,-3.5);
\draw (0,-3.5) node[anchor=north] {6 nm};

\draw[<->] (-3.5,-3) -- (-3.5,3);
\draw (-3.5,0) node[anchor=south,rotate=90] {6 nm};
\draw[<->] (-10.5,-6) -- (-10.5,6);
\draw (-10.5,0) node[anchor=south,rotate=90] {12 nm};

\draw[<->] (-10.5,-10) -- (-10.5,-6);
\draw (-10.5,-8) node[anchor=south,rotate=90] {4 nm};

\draw[<->] (-10.5, 10) -- (-10.5, 6);
\draw (-10.5, 8) node[anchor=south,rotate=90] {4 nm};

\end{tikzpicture}
\end{minipage}
\begin{minipage}{.45\linewidth}
\includegraphics[width=6cm]{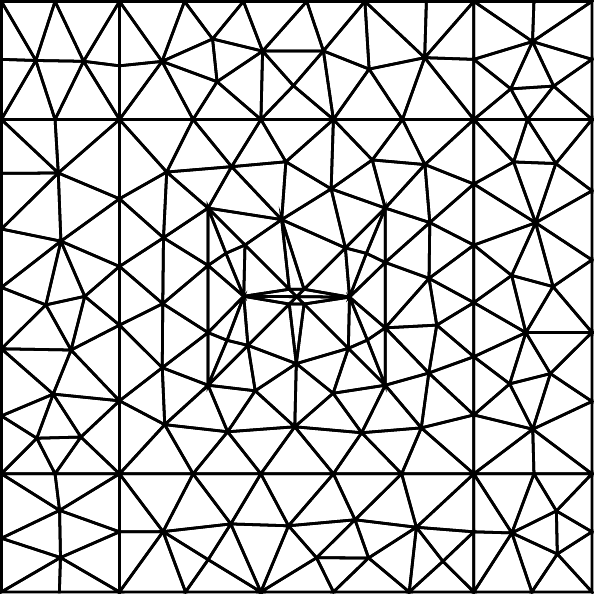}
\end{minipage}

\caption{Settings of the bowtie example (right) and initial mesh for the adaptive algorithm (left).}
\label{figure_bowtie_settings}
\end{figure}
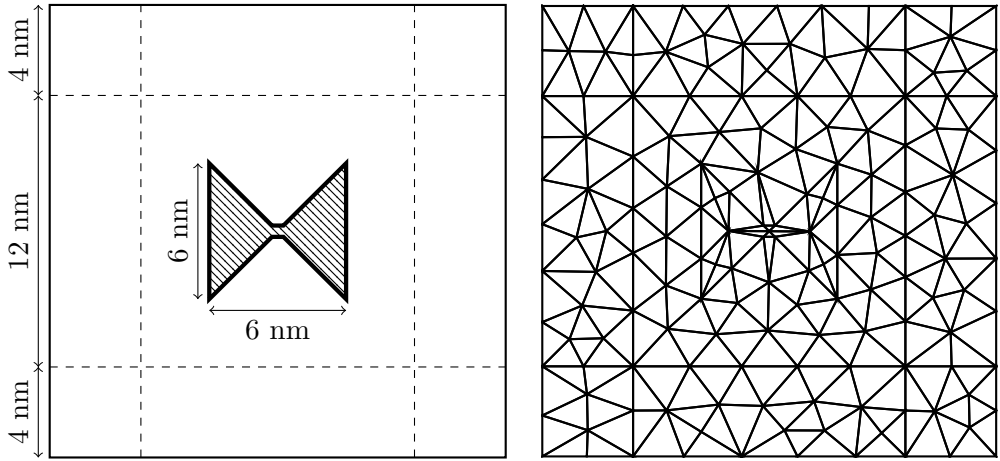

\begin{figure}

\begin{minipage}{.45\linewidth}
\includegraphics[width=\linewidth]{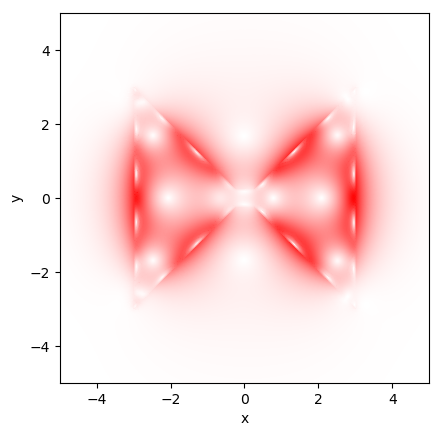}
\subcaption{$\omega = 0.8\omegaP$}
\end{minipage}
\begin{minipage}{.45\linewidth}
\includegraphics[width=\linewidth]{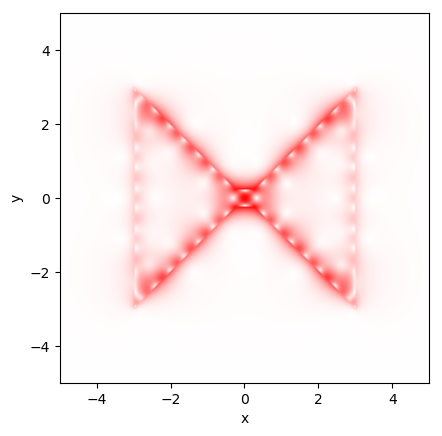}
\subcaption{$\omega = 0.9\omegaP$}
\end{minipage}

\begin{minipage}{.45\linewidth}
\includegraphics[width=\linewidth]{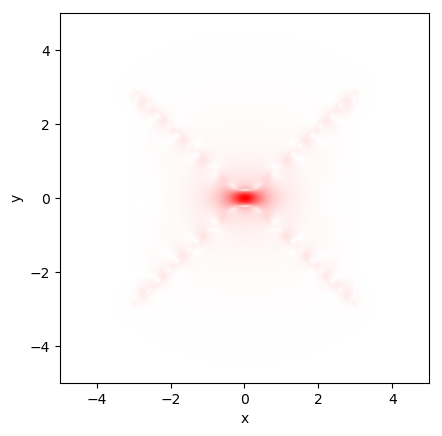}
\subcaption{$\omega = \omegaP$}
\end{minipage}
\begin{minipage}{.45\linewidth}
\caption{Electric field intensities $|\BE|$ in the bowtie experiment}
\label{figure_bowtie_sol}
\end{minipage}
\end{figure}

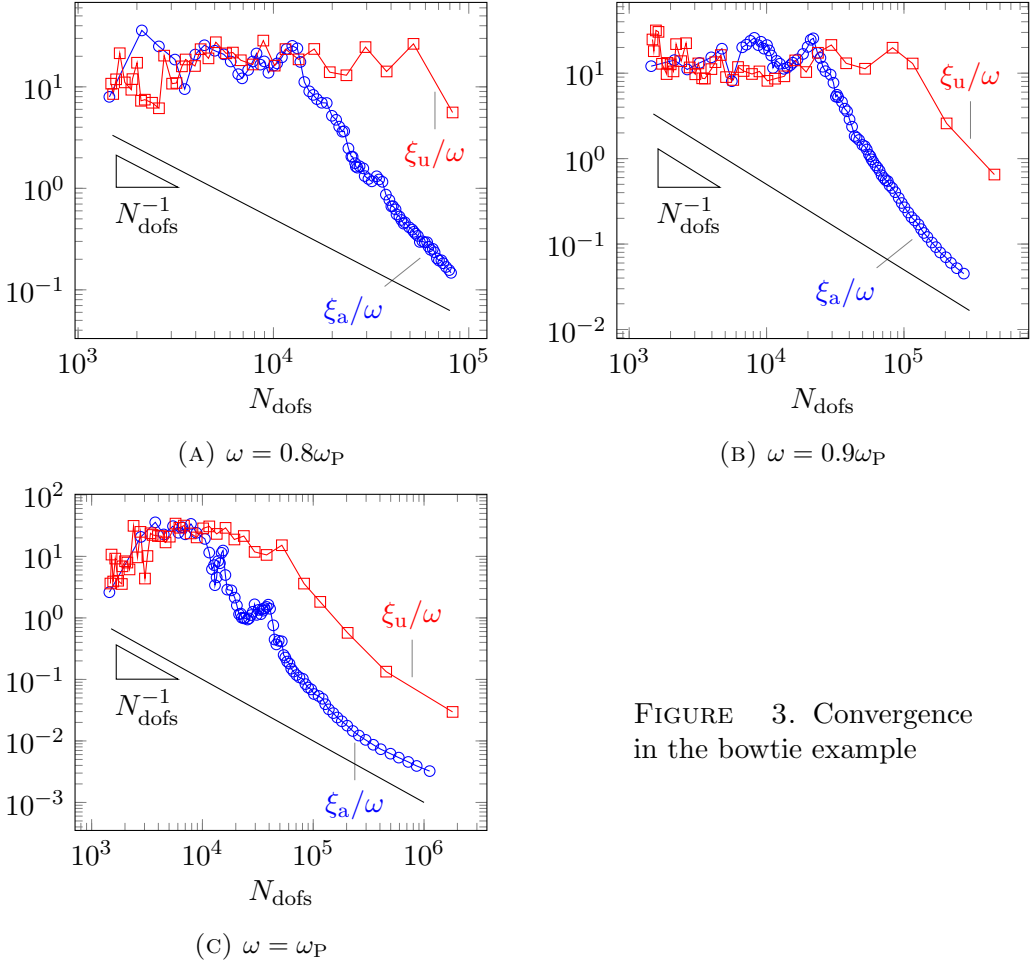
\begin{figure}

\begin{minipage}{.45\linewidth}
\begin{tikzpicture}
\begin{axis}
[
	width  = \linewidth,
	xlabel = {$N_{\rm dofs}$},
	xmode  = log,
	ymode  = log
]

\addplot[color=blue,mark=o     ] table [x = nr_dofs, y = true_error]%
{figures/bowtie/s080/curve.txt} node [pos=.95,pin=-135:{$\xi_{\rm a}/\omega$}] {};

\addplot[color=red,mark=square] table [x = nr_dofs, y = true_error]%
{figures/bowtie/s080/curve_uniform.txt} node [pos=.95,pin=-90:{$\xi_{\rm u}/\omega$}] {};


\plot[domain=1.5e3:8e4] {5e3*x^(-1.)};
\SlopeTriangle{.1}{-.15}{.45}{-1.}{$N_{\rm dofs}^{-1}$}{}

\end{axis}
\end{tikzpicture}
\subcaption{$\omega = 0.8\omegaP$}
\end{minipage}
\begin{minipage}{.45\linewidth}
\begin{tikzpicture}
\begin{axis}
[
	width  = \linewidth,
	xlabel = {$N_{\rm dofs}$},
	xmode  = log,
	ymode  = log
]

\addplot[color=blue,mark=o     ] table [x = nr_dofs, y = true_error]%
{figures/bowtie/s090/curve.txt} node [pos=.9,pin=-135:{$\xi_{\rm a}/\omega$}] {};

\addplot[color=red,mark=square] table [x = nr_dofs, y = true_error]%
{figures/bowtie/s090/curve_uniform.txt} node [pos=.95,pin=90:{$\xi_{\rm u}/\omega$}] {};


\plot[domain=1.5e3:3e5] {.5e4*x^(-1.)};
\SlopeTriangle{.1}{-.15}{.45}{-1.}{$N_{\rm dofs}^{-1}$}{}

\end{axis}
\end{tikzpicture}
\subcaption{$\omega = 0.9\omegaP$}
\end{minipage}

\begin{minipage}{.45\linewidth}
\begin{tikzpicture}
\begin{axis}
[
	width  = \linewidth,
	xlabel = {$N_{\rm dofs}$},
	xmode  = log,
	ymode  = log
]

\addplot[color=blue,mark=o     ] table [x = nr_dofs, y = true_error]%
{figures/bowtie/s100/curve.txt} node [pos=.9,pin=-90:{$\xi_{\rm a}/\omega$}] {};

\addplot[color=red,mark=square] table [x = nr_dofs, y = true_error]%
{figures/bowtie/s100/curve_uniform.txt} node [pos=.95,pin=90:{$\xi_{\rm u}/\omega$}] {};


\plot[domain=1.5e3:1e6] {1e3*x^(-1.)};
\SlopeTriangle{.1}{-.15}{.45}{-1.}{$N_{\rm dofs}^{-1}$}{}

\end{axis}
\end{tikzpicture}
\subcaption{$\omega = \omegaP$}
\end{minipage}
\begin{minipage}{.45\linewidth}
\caption{Convergence in the bowtie example}
\label{figure_bowtie_convergence}
\end{minipage}
\end{figure}

\begin{figure}

\begin{minipage}{.45\linewidth}
\begin{tikzpicture}
\begin{axis}
[
	width  = \linewidth,
	xlabel = {$N_{\rm dofs}$},
	xmode  = log,
	ymode  = log,
	ymax   = 1e3,
	ymin   = 4
]

\addplot[color=blue,mark=o     ] table [x = nr_dofs, y expr = \thisrow{esti_error}/\thisrow{true_error}]%
{figures/bowtie/s080/curve.txt} node [pos=.5,pin=180:{$\eta_{\rm a}/\xi_{\rm a}$}] {};

\addplot[color=red,mark=square] table [x = nr_dofs, y expr = \thisrow{esti_error}/\thisrow{true_error}]%
{figures/bowtie/s080/curve_uniform.txt} node [pos=.8,pin=45:{$\eta_{\rm a}/\xi_{\rm u}$}] {};

\end{axis}
\end{tikzpicture}
\subcaption{$\omega = 0.8\omegaP$}
\end{minipage}
\begin{minipage}{.45\linewidth}
\begin{tikzpicture}
\begin{axis}
[
	width  = \linewidth,
	xlabel = {$N_{\rm dofs}$},
	xmode  = log,
	ymode  = log,
	ymax   = 1e3,
	ymin   = 4
]

\addplot[color=blue,mark=o     ] table [x = nr_dofs, y expr = \thisrow{esti_error}/\thisrow{true_error}]%
{figures/bowtie/s090/curve.txt} node [pos=.6,pin=180:{$\eta_{\rm a}/\xi_{\rm a}$}] {};

\addplot[color=red,mark=square] table [x = nr_dofs, y expr = \thisrow{esti_error}/\thisrow{true_error}]%
{figures/bowtie/s090/curve_uniform.txt} node [pos=.7,pin=45:{$\eta_{\rm a}/\xi_{\rm u}$}] {};

\end{axis}
\end{tikzpicture}
\subcaption{$\omega = 0.9\omegaP$}
\end{minipage}

\begin{minipage}{.45\linewidth}
\begin{tikzpicture}
\begin{axis}
[
	width  = \linewidth,
	xlabel = {$N_{\rm dofs}$},
	xmode  = log,
	ymode  = log,
	ymax   = 1e3,
	ymin   = 4
]

\addplot[color=blue,mark=o     ] table [x = nr_dofs, y expr = \thisrow{esti_error}/\thisrow{true_error}]%
{figures/bowtie/s100/curve.txt} node [pos=.45,pin=180:{$\eta_{\rm a}/\xi_{\rm a}$}] {};

\addplot[color=red,mark=square] table [x = nr_dofs, y expr = \thisrow{esti_error}/\thisrow{true_error}]%
{figures/bowtie/s100/curve_uniform.txt} node [pos=.95,pin=90:{$\eta_{\rm a}/\xi_{\rm u}$}] {};

\end{axis}
\end{tikzpicture}
\subcaption{$\omega = \omegaP$}
\end{minipage}
\begin{minipage}{.45\linewidth}
\caption{Effectivity indices in the bowtie example}
\label{figure_bowtie_effectivity}
\end{minipage}
\end{figure}
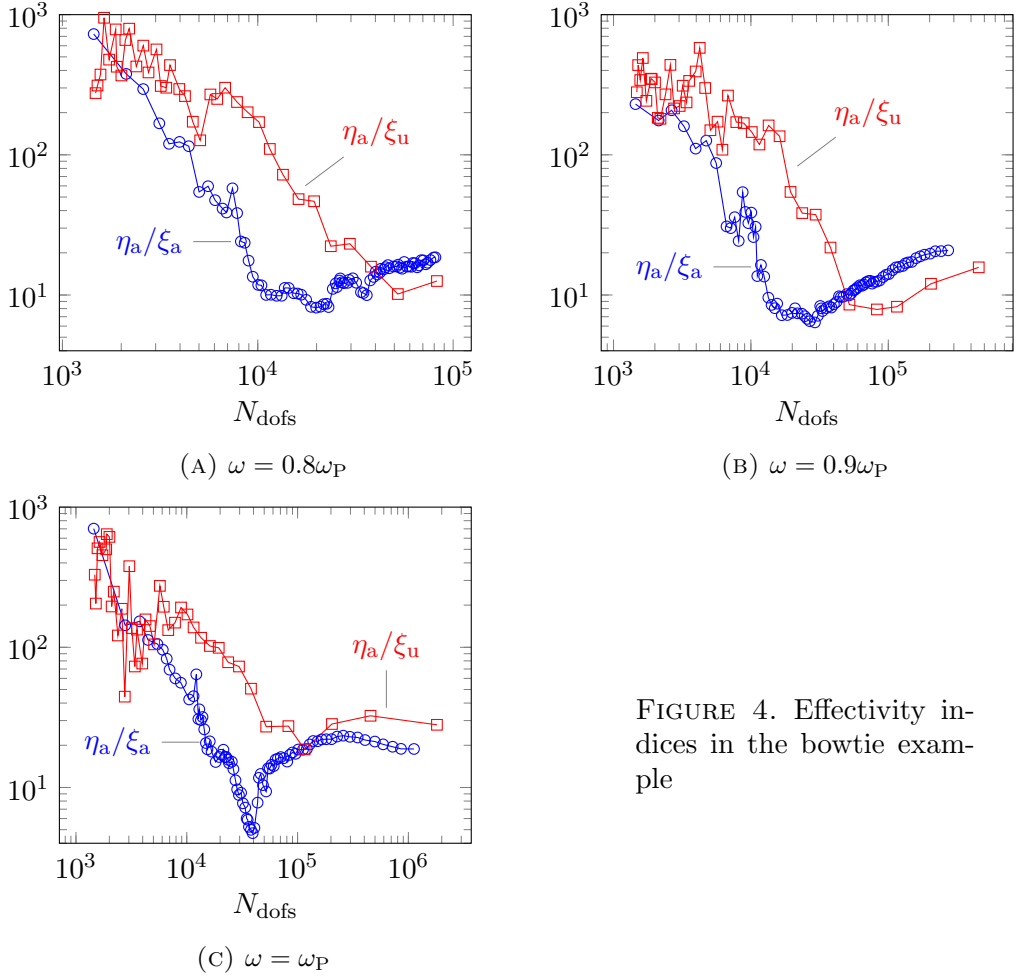

\begin{figure}

\begin{minipage}{.45\linewidth}
\includegraphics[width=\linewidth]{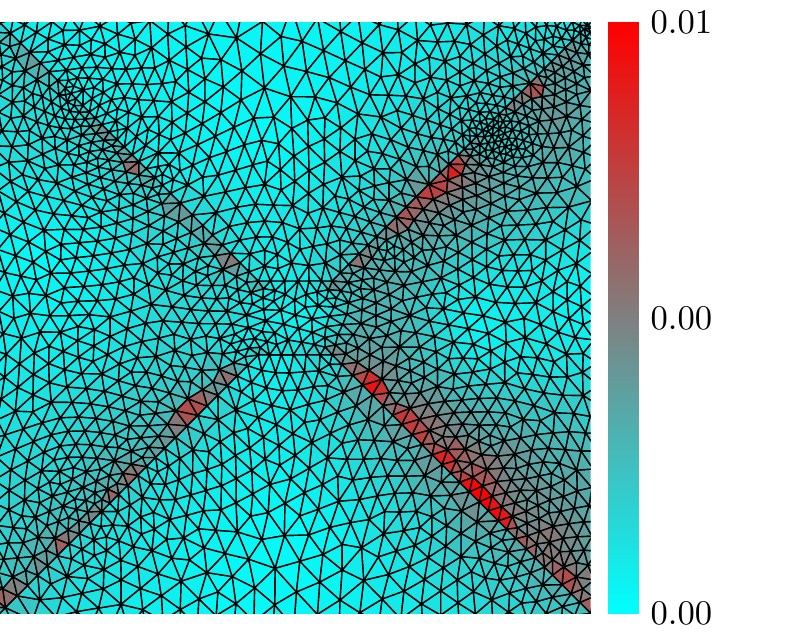}
\end{minipage}
\begin{minipage}{.45\linewidth}
\includegraphics[width=\linewidth]{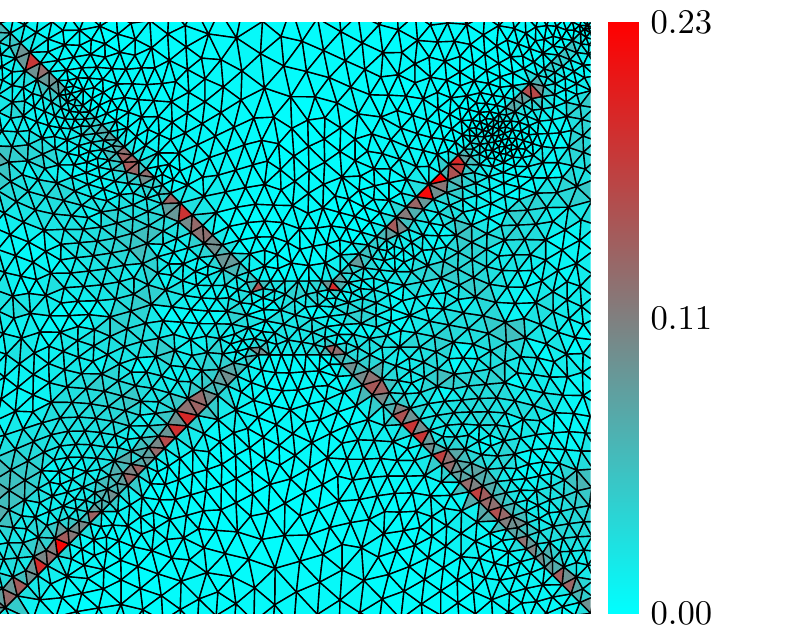}
\end{minipage}
\caption{Actual (left panel) and estimated (right panel) errors at the iteration
\#80 of the adaptive algorithm for the bowtie antenna example with $\omega = 0.8\omegaP$.}
\label{figure_el_bowtie_s080}
\end{figure}

\subsection{Silver nanotip}

Here, we model the silver  nanotip depicted  on the left  panel of
Figure \ref{figure_triangle_settings}.  We consider  three frequencies
of   interest,  namely   $\omega   \eq   0.7\omegaP$,  $\omegaP$   and
$1.3\omegaP$.  In every  case, we  select an  incident planewave  with
direction $\bd = (1,0)$, and polarization  $\bp = (0,1)$, and we begin
the adaptive  algorithm on the  initial mesh represented in  the right
panel of Figure  \ref{figure_triangle_settings}. The discrete solution
is computed  with a polynomial  degree $p=2$  and we run  the adaptive
loop for 50 iterations. The reference solutions computed on the finest
meshes are presented in Figure \ref{figure_triangle_sol}.

Figure \ref{figure_triangle_convergence} shows the convergence history
for  adaptive  and uniform  meshes.  The  adaptive meshes  drastically
improve the accuracy,  and yield the optimal  convergence rate $N_{\rm
  dofs}^{-(p+1)/2}$.   The   case  where  $\omega  =   0.7\omegaP$  is
particularly  instructive,  since   uniform  meshes  clearly  converge
suboptimally.   We  present   the   effectivity   indices  in   Figure
\ref{figure_triangle_effectivity}. As previously  stated, we observe a
usual behaviour, which is in agreement with previous works and our key
theoretical   results.    In  Figure   \ref{fig_el_nanotip_s130},   we
represent the  elementwise error distribution  and the estimator  in a
neighborhood  of  the  nanotip  at  iteration  \#25  of  the  adaptive
algorithm. We observe  a nice agreement between the  estimator and the
actual   error  for   the   selected   frequencies.  Finally,   Figure
\ref{fig_fm_nanotip_s130}  features the  final  mesh  produced by  the
adaptive algorithm at the last  iteration (\#50). The meshes are finer
close to the  inclusion, with specific refinements close  to the edges
and corners of the tip, as to be expected.

\begin{figure}

\begin{minipage}{.45\linewidth}
\begin{tikzpicture}[scale=.3]

\draw[ultra thick,pattern=north west lines, pattern color=black] (-3,-1) -- (3,0) -- (-3,1) -- cycle;

\draw[thick] (-10,-10) rectangle (10,10);

\draw[dashed] (- 6,-10) -- (- 6, 10);
\draw[dashed] (  6,-10) -- (  6, 10);
\draw[dashed] (-10,- 6) -- ( 10,- 6);
\draw[dashed] (-10,  6) -- ( 10,  6);

\draw[<->] (-3,-1.5) -- (3,-1.5);
\draw (0,-1.5) node[anchor=north] {6 nm};

\draw[<->] (-3.5,-1) -- (-3.5,1);
\draw (-3.5,0) node[anchor=south,rotate=90] {2 nm};

\draw[<->] (-10.5,-6) -- (-10.5,6);
\draw (-10.5,0) node[anchor=south,rotate=90] {12 nm};

\draw[<->] (-10.5,-10) -- (-10.5,-6);
\draw (-10.5,-8) node[anchor=south,rotate=90] {4 nm};

\draw[<->] (-10.5, 10) -- (-10.5, 6);
\draw (-10.5, 8) node[anchor=south,rotate=90] {4 nm};

\end{tikzpicture}
\end{minipage}
\begin{minipage}{.45\linewidth}
\includegraphics[width=6cm]{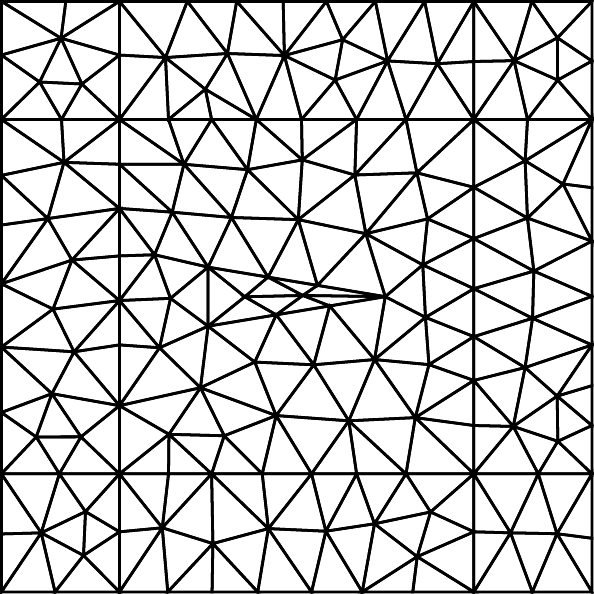}
\end{minipage}

\caption{Settings of the nanotip example (right) and initial mesh for the adaptive algorithm (left).}
\label{figure_triangle_settings}
\end{figure}
 
\begin{figure}

\begin{minipage}{.45\linewidth}
\includegraphics[width=\linewidth]{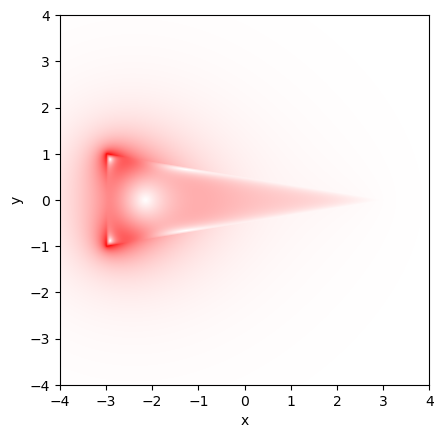}
\subcaption{$\omega = 0.7\omegaP$}
\end{minipage}
\begin{minipage}{.45\linewidth}
\includegraphics[width=\linewidth]{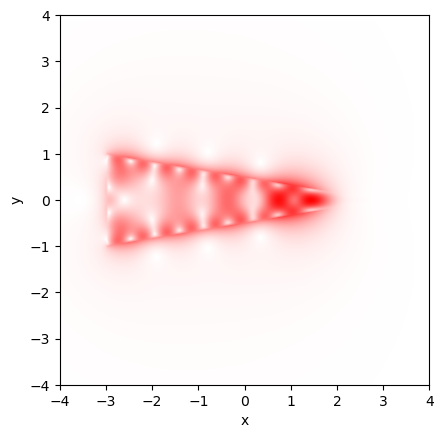}
\subcaption{$\omega = \omegaP$}
\end{minipage}

\begin{minipage}{.45\linewidth}
\includegraphics[width=\linewidth]{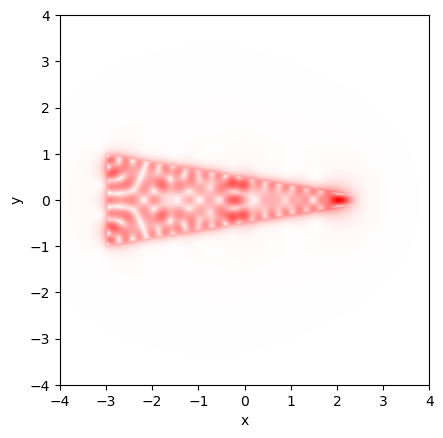}
\subcaption{$\omega = 1.3\omegaP$}
\end{minipage}
\begin{minipage}{.45\linewidth}
\caption{Electric field intensities $|\BE|$ in the nanotip experiment}
\label{figure_triangle_sol}
\end{minipage}
\end{figure}

\begin{figure}
\begin{minipage}{.45\linewidth}
\begin{tikzpicture}
\begin{axis}
[
	width  = \linewidth,
	xlabel = {$N_{\rm dofs}$},
	xmode  = log,
	ymode  = log
]

\addplot[color=blue,mark=o     ] table [x = nr_dofs, y = true_error]%
{figures/triangle/s070/curve.txt} node [pos=.5,pin=-90:{$\xi_{\rm a}/\omega$}] {};

\addplot[color=red,mark=square] table [x = nr_dofs, y = true_error]%
{figures/triangle/s070/curve_uniform.txt} node [pos=.8,pin=90:{$\xi_{\rm u}/\omega$}] {};


\plot[domain=3.e3:1e5] {1e4*x^(-3./2.)};

\SlopeTriangle{.1}{-.15}{.25}{-3./2.}{$N_{\rm dofs}^{-3/2}$}{}

\end{axis}
\end{tikzpicture}
\subcaption{$\omega = 0.7\omegaP$}
\end{minipage}
\begin{minipage}{.45\linewidth}
\begin{tikzpicture}
\begin{axis}
[
	width  = \linewidth,
	xlabel = {$N_{\rm dofs}$},
	xmode  = log,
	ymode  = log
]

\addplot[color=blue,mark=o     ] table [x = nr_dofs, y = true_error]%
{figures/triangle/s100/curve.txt} node [pos=.8,pin=-90:{$\xi_{\rm a}/\omega$}] {};

\addplot[color=red,mark=square] table [x = nr_dofs, y = true_error]%
{figures/triangle/s100/curve_uniform.txt} node [pos=.9,pin=-90:{$\xi_{\rm u}/\omega$}] {};


\plot[domain=3.e3:9e4] {8e4*x^(-3./2.)};

\SlopeTriangle{.1}{-.15}{.25}{-3./2.}{$N_{\rm dofs}^{-3/2}$}{}

\end{axis}
\end{tikzpicture}
\subcaption{$\omega = \omegaP$}
\end{minipage}

\begin{minipage}{.45\linewidth}
\begin{tikzpicture}
\begin{axis}
[
	width  = \linewidth,
	xlabel = {$N_{\rm dofs}$},
	xmode  = log,
	ymode  = log
]

\addplot[color=blue,mark=o     ] table [x = nr_dofs, y = true_error]%
{figures/triangle/s130/curve.txt} node [pos=.8,pin=-90:{$\xi_{\rm a}/\omega$}] {};

\addplot[color=red,mark=square ] table [x = nr_dofs, y = true_error]%
{figures/triangle/s130/curve_uniform.txt} node [pos=.9,pin=-90:{$\xi_{\rm u}/\omega$}] {};


\plot[domain=3.e3:3e5] {5e4*x^(-3./2.)};

\SlopeTriangle{.1}{-.15}{.3}{-3./2.}{$N_{\rm dofs}^{-3/2}$}{}

\end{axis}
\end{tikzpicture}
\subcaption{$\omega = 1.3\omegaP$}
\end{minipage}
\begin{minipage}{.45\linewidth}
\caption{Convergence in the nanotip example}
\label{figure_triangle_convergence}
\end{minipage}
\end{figure}
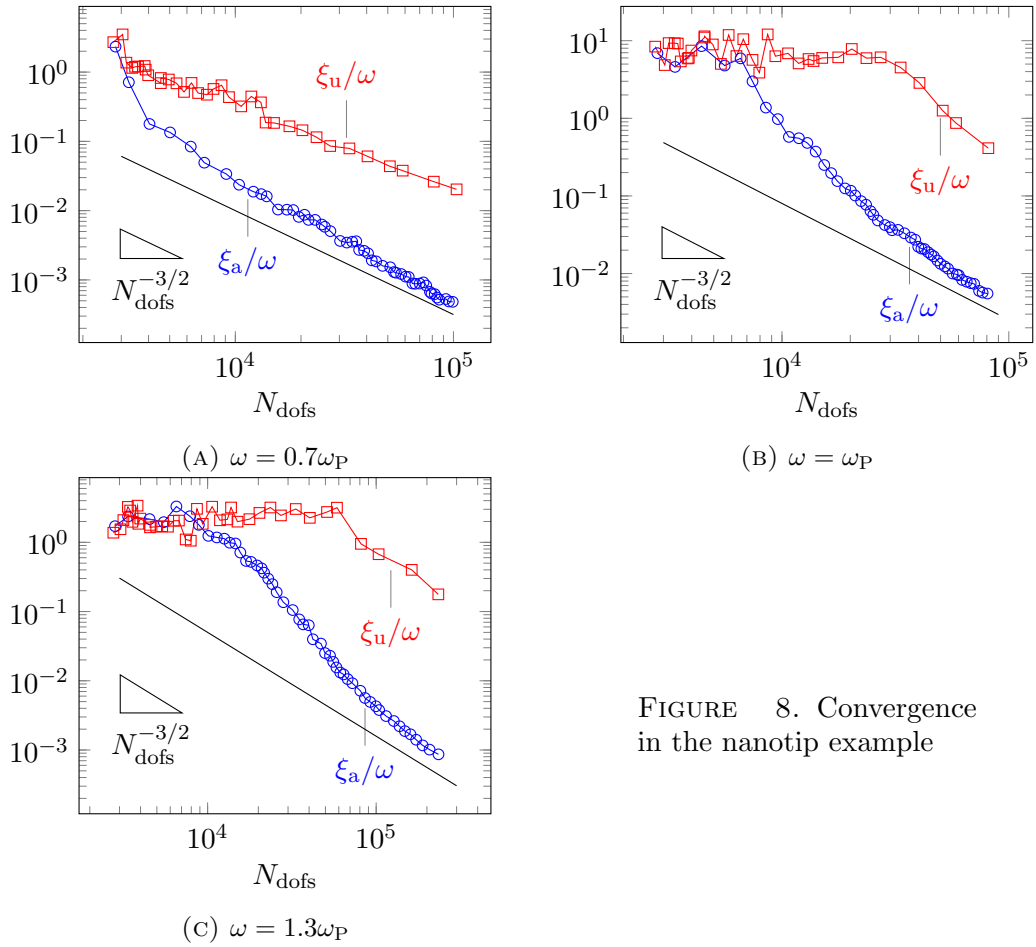

\begin{figure}
\begin{minipage}{.45\linewidth}
\begin{tikzpicture}
\begin{axis}
[
	width  = \linewidth,
	xlabel = {$N_{\rm dofs}$},
	xmode  = log,
	ymode  = log
]

\addplot[color=blue,mark=o     ] table [x = nr_dofs, y expr = \thisrow{esti_error}/\thisrow{true_error}]%
{figures/triangle/s070/curve.txt} node [pos=.3,pin=-90:{$\eta_{\rm a}/\xi_{\rm a}$}] {};

\addplot[color=red,mark=square] table [x = nr_dofs, y expr = \thisrow{esti_error}/\thisrow{true_error}]%
{figures/triangle/s070/curve_uniform.txt} node [pos=.8,pin=90:{$\eta_{\rm u}/\xi_{\rm u}$}] {};

\end{axis}
\end{tikzpicture}
\subcaption{$\omega = 0.7\omegaP$}
\end{minipage}
\begin{minipage}{.45\linewidth}
\begin{tikzpicture}
\begin{axis}
[
	width  = \linewidth,
	xlabel = {$N_{\rm dofs}$},
	xmode  = log,
	ymode  = log
]

\addplot[color=blue,mark=o     ] table [x = nr_dofs, y expr = \thisrow{esti_error}/\thisrow{true_error}]%
{figures/triangle/s100/curve.txt} node [pos=.85,pin=90:{$\eta_{\rm a}/\xi_{\rm a}$}] {};

\addplot[color=red,mark=square] table [x = nr_dofs, y expr = \thisrow{esti_error}/\thisrow{true_error}]%
{figures/triangle/s100/curve_uniform.txt} node [pos=.7,pin=0:{$\eta_{\rm u}/\xi_{\rm u}$}] {};

\end{axis}
\end{tikzpicture}
\subcaption{$\omega = \omegaP$}
\end{minipage}

\begin{minipage}{.45\linewidth}
\begin{tikzpicture}
\begin{axis}
[
	width  = \linewidth,
	xlabel = {$N_{\rm dofs}$},
	xmode  = log,
	ymode  = log
]

\addplot[color=blue,mark=o     ] table [x = nr_dofs, y expr = \thisrow{esti_error}/\thisrow{true_error}]%
{figures/triangle/s130/curve.txt} node [pos=.9,pin=90:{$\eta_{\rm a}/\xi_{\rm a}$}] {};

\addplot[color=red,mark=square] table [x = nr_dofs, y expr = \thisrow{esti_error}/\thisrow{true_error}]%
{figures/triangle/s130/curve_uniform.txt} node [pos=.65,pin=0:{$\eta_{\rm u}/\xi_{\rm u}$}] {};

\end{axis}
\end{tikzpicture}
\subcaption{$\omega = 1.3\omegaP$}
\end{minipage}
\begin{minipage}{.45\linewidth}
\caption{Effectivity in the nanotip example}
\label{figure_triangle_effectivity}
\end{minipage}
\end{figure}
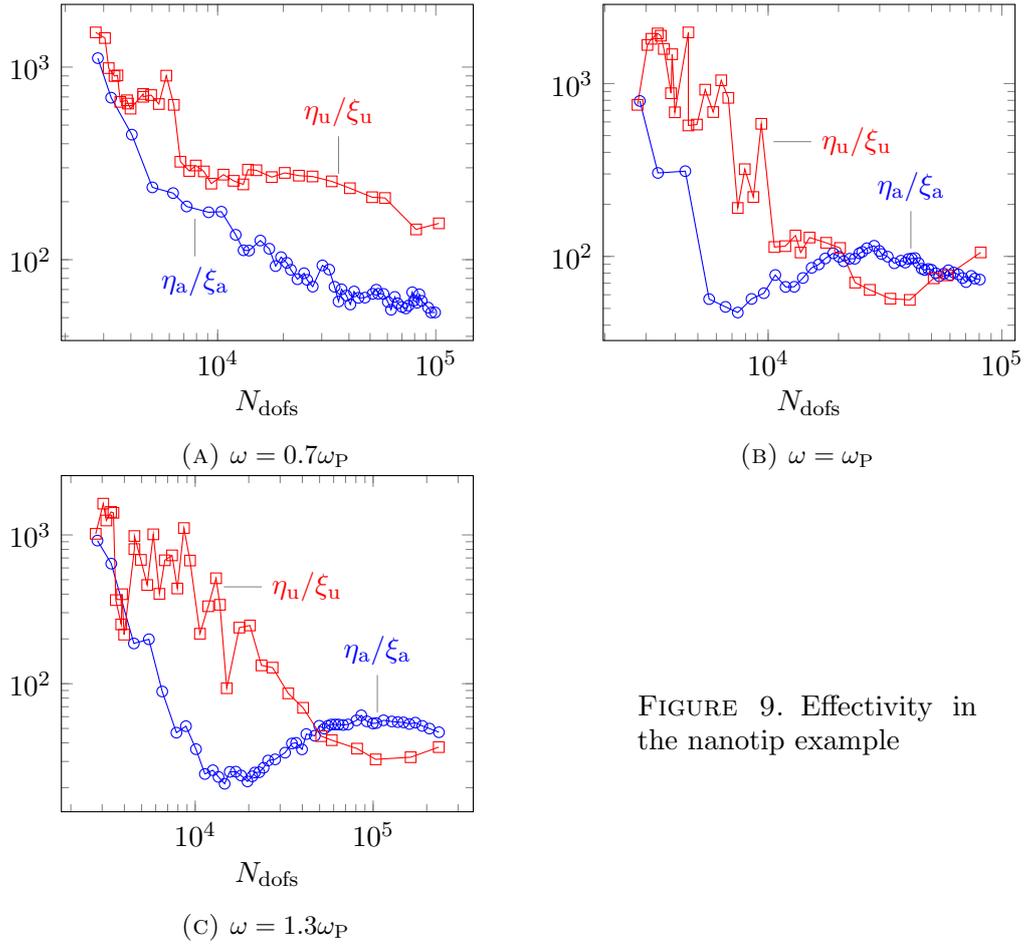

\begin{figure}
\centering
\includegraphics[width=0.75\linewidth]{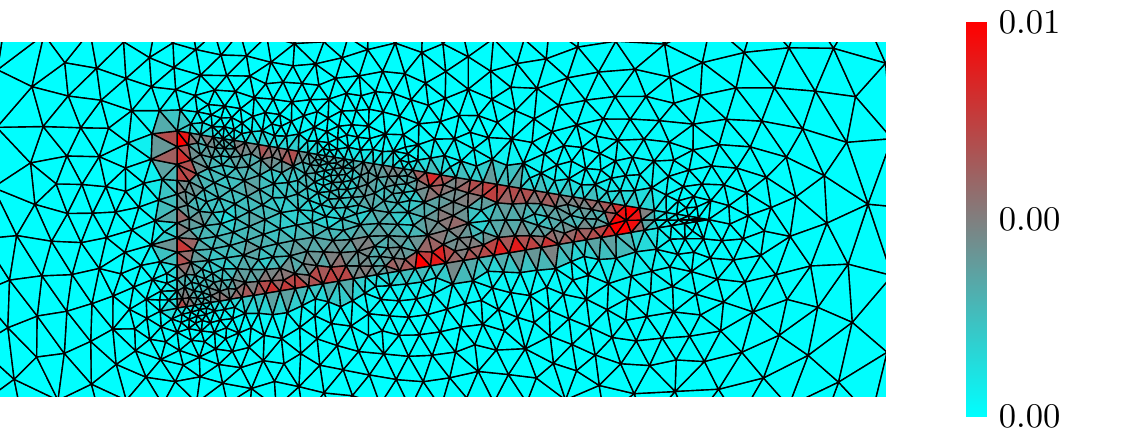}
\includegraphics[width=0.75\linewidth]{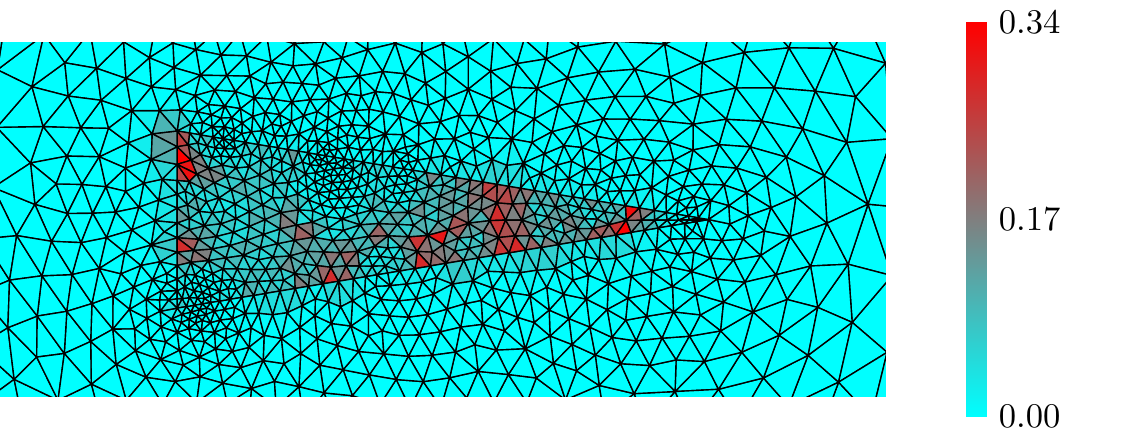}
\caption{Actual (top panel) and estimated (bottom panel) errors at the iteration
\#25 of the adaptive algorithm for the nanotip example with $\omega = 1.3\omegaP$.}
\label{fig_el_nanotip_s130}
\end{figure}

\begin{figure}
\centering
\begin{minipage}{.35\linewidth}
\includegraphics[width=\linewidth]{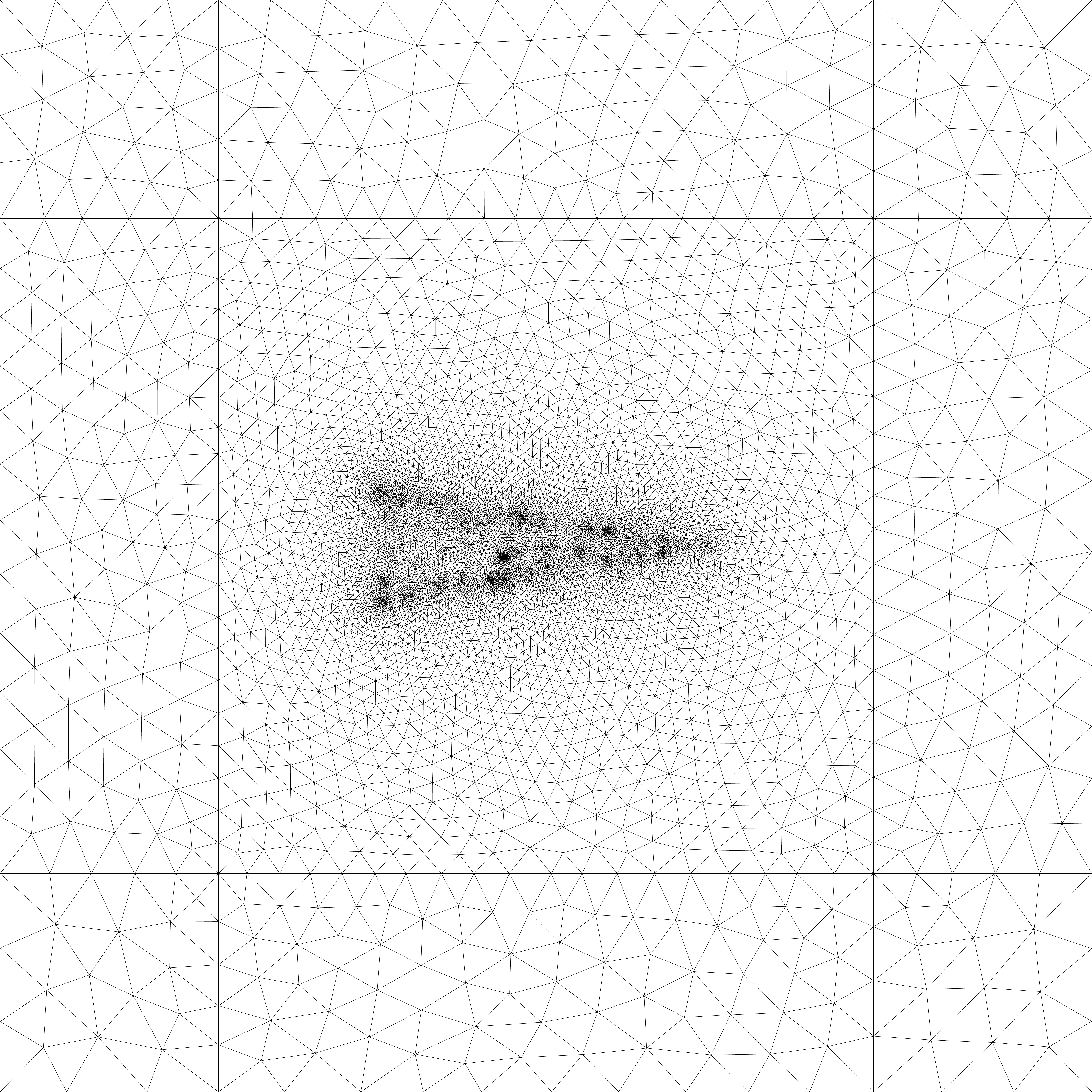}
\end{minipage}
\begin{minipage}{.55\linewidth}
\includegraphics[width=\linewidth]{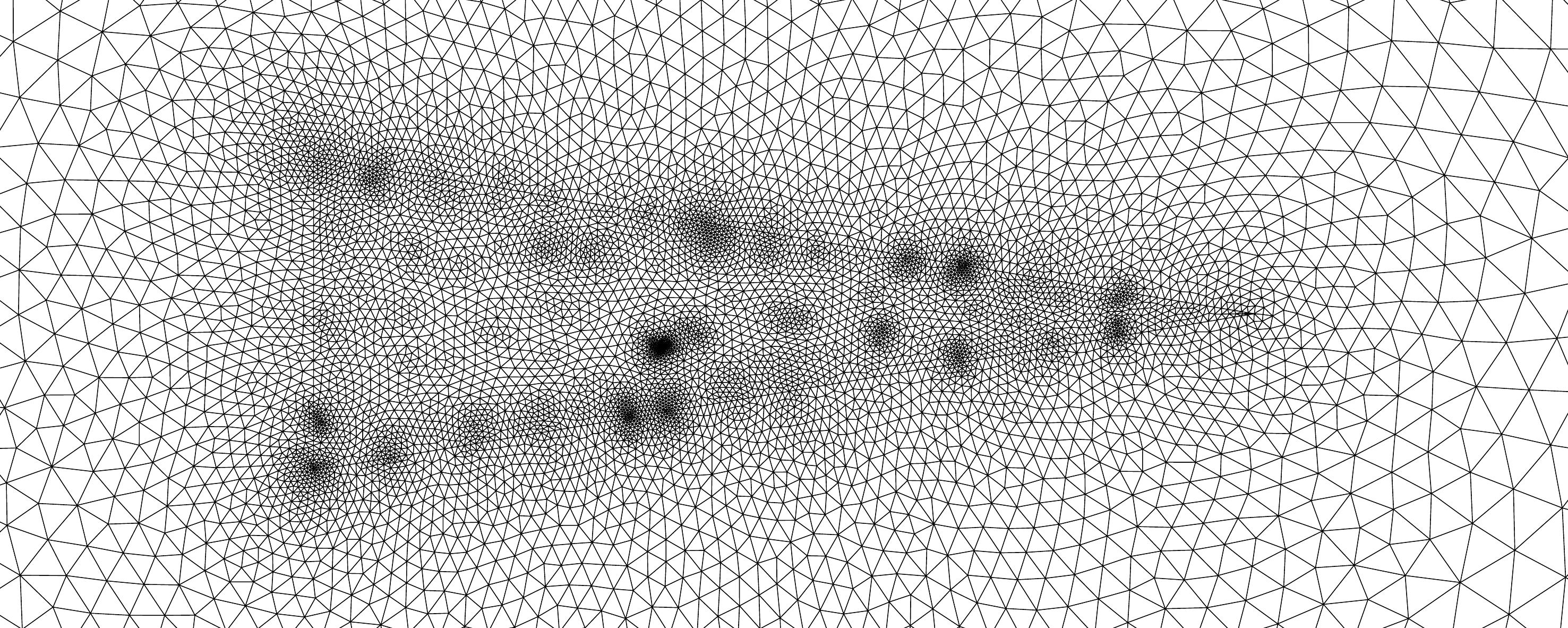}
\end{minipage}

\caption{Computational mesh at the \#50 iteration of the adaptive algorithm for
the nanotip example with $\omega = 1.3\omegaP$. The right panel presents a focus
on the inclusion.}
\label{fig_fm_nanotip_s130}
\end{figure}

\subsection{Gold V-groove channel}

The last example is a section of a ``V-groove'' channel depicted in Figure
\ref{figure_groove_settings}. The incidence angle is again $\theta = \pi/3$,
with $\bd = (\cos\theta,\sin\theta)$ and $\bp = (-\sin\theta,\cos\theta)$.
The reference solutions produced on the finest
meshes are presented in Figure \ref{figure_groove_sol} for $\omega=0.8\omegaP$,
$0.9\omegaP$ and $\omegaP$. The desired behaviour is observed in the case
$\omega=0.9\omegaP$ where the electric field is localized in the ``V'' cavity,
which can be used to design a waveguide along the transverse direction. We run the
adaptive loop for 50 iterations starting with the initial mesh of Figure
\ref{figure_groove_settings} and $p=3$.

As in the other experiments, Figure \ref{figure_groove_convergence} presents
the behaviour of the actual error against the number of degrees of freedom,
and we observe a large accuracy enhancement on adaptive meshes, together
with an optimal convergence rate. The effectivity indices are represented on
Figure \ref{figure_groove_effectivity}. They exhibit
a nicer behaviour than in the previous experiments. This is linked to the
fact that a higher polynomial degree is employed with similar starting
mesh sizes, which shorten the ``pre-asymptotic regime'' where the reliability
constant may depend on the mesh size. Figure \ref{figure_el_groove_s090} shows
the actual and estimated error distributions are very similar, again
illustrating the quality of the proposed estimator.

\input{figures/groove/settings}
\begin{figure}

\begin{minipage}{.45\linewidth}
\includegraphics[width=\linewidth]{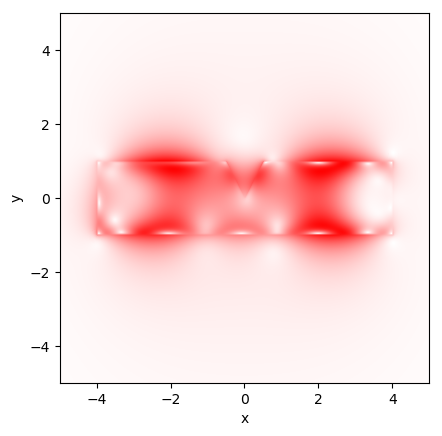}
\subcaption{$\omega = 0.8\omegaP$}
\end{minipage}
\begin{minipage}{.45\linewidth}
\includegraphics[width=\linewidth]{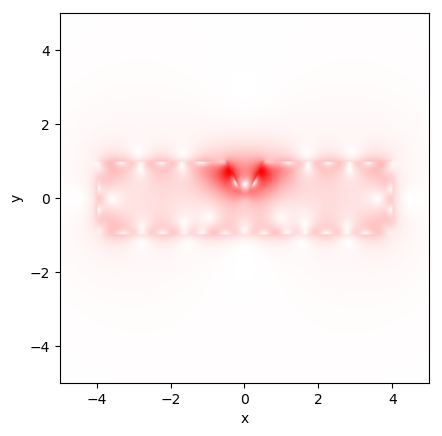}
\subcaption{$\omega = 0.9\omegaP$}
\end{minipage}

\begin{minipage}{.45\linewidth}
\includegraphics[width=\linewidth]{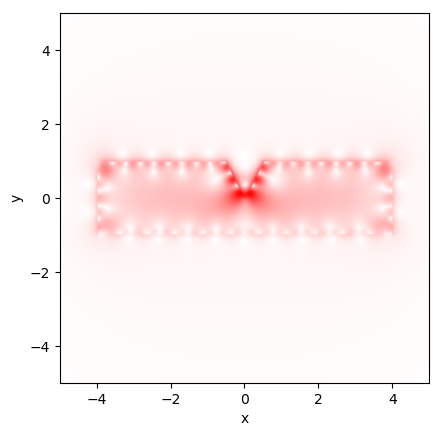}
\subcaption{$\omega = \omegaP$}
\end{minipage}
\begin{minipage}{.45\linewidth}
\caption{Electric field intensities $|\BE|$ in the V-groove experiment}
\label{figure_groove_sol}
\end{minipage}
\end{figure}

\begin{figure}
\begin{minipage}{.45\linewidth}
\begin{tikzpicture}
\begin{axis}
[
	xlabel = {$N_{\rm dofs}$},
	width=\linewidth,
	xmode  = log,
	ymode  = log
]

\addplot[color=blue,mark=o     ] table [x = nr_dofs, y = true_error]%
{figures/groove/s080/curve.txt} node [pos=.8,pin=-90:{$\xi_{\rm a}/\omega$}] {};

\addplot[color=red,mark=square] table [x = nr_dofs, y = true_error]%
{figures/groove/s080/curve_uniform.txt} node [pos=.8,pin=90:{$\xi_{\rm u}/\omega$}] {};


\plot[domain=1.8e4:1.2e5] {3.e7*x^(-2.)};
\SlopeTriangle{.1}{-.15}{.2}{-2.}{$N_{\rm dofs}^{-2}$}{}

\end{axis}
\end{tikzpicture}
\subcaption{$\omega = 0.8\omegaP$}
\end{minipage}
\begin{minipage}{.45\linewidth}
\begin{tikzpicture}
\begin{axis}
[
	xlabel = {$N_{\rm dofs}$},
	width=\linewidth,
	xmode  = log,
	ymode  = log
]

\addplot[color=blue,mark=o     ] table [x = nr_dofs, y = true_error]%
{figures/groove/s090/curve.txt} node [pos=.7,pin=-90:{$\xi_{\rm a}/\omega$}] {};

\addplot[color=red,mark=square] table [x = nr_dofs, y = true_error]%
{figures/groove/s090/curve_uniform.txt} node [pos=.9,pin=90:{$\xi_{\rm u}/\omega$}] {};


\plot[domain=1.8e4:1.e5] {5.e7*x^(-2.)};
\SlopeTriangle{.1}{-.15}{.2}{-2.}{$N_{\rm dofs}^{-2}$}{}

\end{axis}
\end{tikzpicture}
\subcaption{$\omega = 0.9\omegaP$}
\end{minipage}

\begin{minipage}{.45\linewidth}
\begin{tikzpicture}
\begin{axis}
[
	width=\linewidth,
	xlabel = {$N_{\rm dofs}$},
	xmode  = log,
	ymode  = log
]

\addplot[color=blue,mark=o     ] table [x = nr_dofs, y = true_error]%
{figures/groove/s100/curve.txt} node [pos=.8,pin=-90:{$\xi_{\rm a}/\omega$}] {};

\addplot[color=red,mark=square] table [x = nr_dofs, y = true_error]%
{figures/groove/s100/curve_uniform.txt} node [pos=.8,pin=-90:{$\xi_{\rm u}/\omega$}] {};


\plot[domain=1.8e4:1.e5] {5.e7*x^(-2.)};
\SlopeTriangle{.1}{-.15}{.2}{-2.}{$N_{\rm dofs}^{-2}$}{}

\end{axis}
\end{tikzpicture}
\subcaption{$\omega = \omegaP$}
\end{minipage}
\begin{minipage}{.45\linewidth}
\caption{Convergence history of the adaptive algorithm for the V-groove example}
\label{figure_groove_convergence}
\end{minipage}
\end{figure}

\begin{figure}
\begin{minipage}{.45\linewidth}
\begin{tikzpicture}
\begin{axis}
[
	xlabel = {$N_{\rm dofs}$},
	width=\linewidth,
	xmode  = log,
	ymode  = log,
	ymax   = 1e3,
	ymin   = 10
]

\addplot[color=blue,mark=o     ] table [x = nr_dofs, y expr = \thisrow{esti_error}/\thisrow{true_error}]%
{figures/groove/s080/curve.txt} node [pos=.6,pin=90:{$\eta_{\rm a}/\xi_{\rm a}$}] {};

\addplot[color=red,mark=square ] table [x = nr_dofs, y expr = \thisrow{esti_error}/\thisrow{true_error}]%
{figures/groove/s080/curve_uniform.txt} node [pos=.6,pin=-90:{$\eta_{\rm u}/\xi_{\rm u}$}] {};

\end{axis}
\end{tikzpicture}
\subcaption{$\omega = 0.8\omegaP$}
\end{minipage}
\begin{minipage}{.45\linewidth}
\begin{tikzpicture}
\begin{axis}
[
	xlabel = {$N_{\rm dofs}$},
	width=\linewidth,
	xmode  = log,
	ymode  = log,
	ymax   = 1e3,
	ymin   = 10
]

\addplot[color=blue,mark=o     ] table [x = nr_dofs, y expr = \thisrow{esti_error}/\thisrow{true_error}]%
{figures/groove/s090/curve.txt} node [pos=.6,pin=90:{$\eta_{\rm a}/\xi_{\rm a}$}] {};

\addplot[color=red,mark=square ] table [x = nr_dofs, y expr = \thisrow{esti_error}/\thisrow{true_error}]%
{figures/groove/s090/curve_uniform.txt} node [pos=.6,pin=-90:{$\eta_{\rm u}/\xi_{\rm u}$}] {};

\end{axis}
\end{tikzpicture}
\subcaption{$\omega = 0.9\omegaP$}
\end{minipage}

\begin{minipage}{.45\linewidth}
\begin{tikzpicture}
\begin{axis}
[
	width=\linewidth,
	xlabel = {$N_{\rm dofs}$},
	xmode  = log,
	ymode  = log,
	ymax   = 1e3,
	ymin   = 10
]

\addplot[color=blue,mark=o     ] table [x = nr_dofs, y expr = \thisrow{esti_error}/\thisrow{true_error}]%
{figures/groove/s100/curve.txt} node [pos=.8,pin=90:{$\eta_{\rm a}/\xi_{\rm a}$}] {};

\addplot[color=red,mark=square ] table [x = nr_dofs, y expr = \thisrow{esti_error}/\thisrow{true_error}]%
{figures/groove/s100/curve_uniform.txt} node [pos=.9,pin=-90:{$\eta_{\rm u}/\xi_{\rm u}$}] {};

\end{axis}
\end{tikzpicture}
\subcaption{$\omega = \omegaP$}
\end{minipage}
\begin{minipage}{.45\linewidth}
\caption{Convergence history of the adaptive algorithm for the V-groove example}
\label{figure_groove_effectivity}
\end{minipage}
\end{figure}
\begin{figure}
\centering
\includegraphics[width=0.7\linewidth]{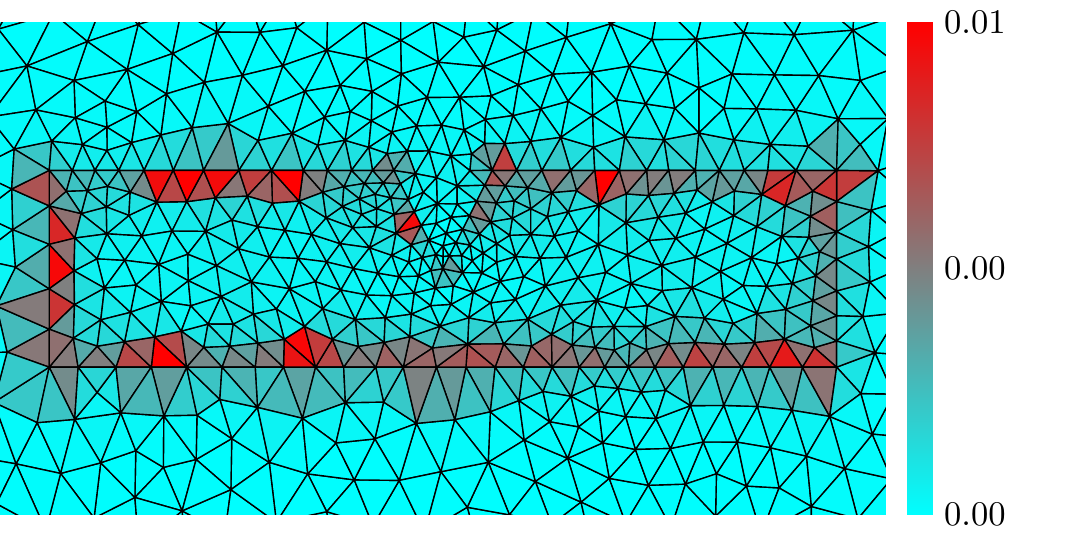}
\includegraphics[width=0.7\linewidth]{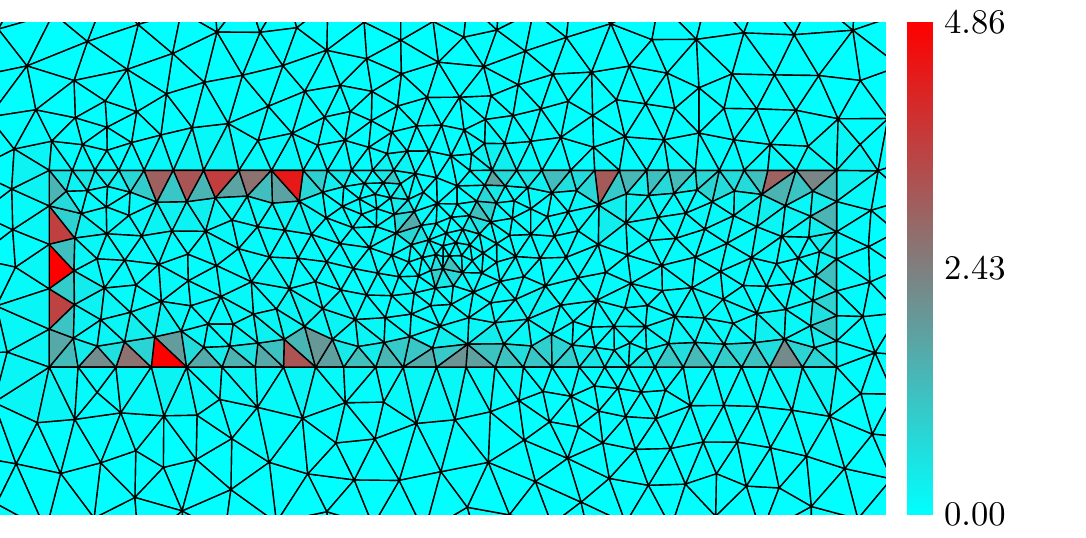}
\caption{Actual (top panel) and estimated (bottom panel) errors at the iteration
\#80 of the adaptive algorithm for the V-groove example with $\omega = 0.9\omegaP$.}
\label{figure_el_groove_s090}
\end{figure}

\section{Conclusion}
\label{section_conclusions}

We propose a novel residual-based {\it a posteriori} error
estimator for finite element discretizations of Maxwell's equations
coupled with a non-local hydrodynamic Drude model taking into account
spatial dispersion effects in metallic nanostructures. At the theoretical
level, we establish reliability  and efficiency of the estimator. We also
propose a number of relevant two-dimensional examples where the error estimator
drives an adaptive procedure. We observe the expected optimal convergence
rate meaning that the estimator correctly steers the adaptive process. Besides,
the adaptive algorithm enables substantial computational savings, as compared to
the use of uniform meshes. These preliminary results are very promising, and future
work will focus on more realistic three-dimensional benchmarks.

\bibliographystyle{amsplain}
\bibliography{bibliography.bib}

\end{document}